\documentclass{article}
\usepackage{latexsym}
\usepackage{amsmath}
\usepackage{amsfonts}
\usepackage{amssymb}
\usepackage{amsthm}
\usepackage{amscd}
\usepackage[all]{xy}
\usepackage{color}

\newtheorem{theorem}{Theorem}[section]

\newtheorem{corollary}[theorem]{Corollary}
\newtheorem{lemma}[theorem]{Lemma}
\newtheorem{example}[theorem]{Example}
\newtheorem{definition}[theorem]{Definition}
\newtheorem{proposition}[theorem]{Proposition}
\newtheorem{remark}[theorem]{Remark}
\newtheorem{comments }[theorem]{Comments}
\numberwithin{theorem}{section}
\input{tcilatex}
\begin{document}

\title{A construction of pro-$C^{\ast }$-algebras from pro-$C^{\ast }$%
-correspondences}
\author{Maria Joi\c{t}a and Ioannis Zarakas}
\maketitle

\begin{abstract}
We associate a pro-$C^{\ast }$-algebra to a pro-$C^{\ast }$-correspondence
and show that this construction generalizes the construction of crossed
products by Hilbert pro-$C^{\ast }$-bimodules and the construction of pro-$%
C^{\ast }$-crossed products by strong bounded automorphisms.
\end{abstract}

\begin{center}
\emph{This paper is dedicated to the memory of Anastasios Mallios}
\end{center}

\medskip

\footnotetext{%
} \footnotetext{%
}

\section{Introduction}

The notion of a Hilbert $C^{\ast }$-module has first been introduced by I.
Kaplansky in 1953. It is a generalization of a Hilbert space, in the sense
that the inner product in a Hilbert $C^{\ast }$-module takes values in a $%
C^{\ast }$-algebra. Since 1953, there has been a continuous development of
the theory of Hilbert $C^{\ast }$-modules which has offered a very rich
literature and useful tools for various important fields of mathematics,
such as KK-theory, $C^{\ast }$-algebraic quantum group theory and groupoid $%
C^{\ast }$-algebras.

A $C^{\ast }$-correspondence is a natural generalization of a Hilbert C$%
^{\ast }$-bimodule. Namely it is a pair $(X,A)$, where $X$ is a right
Hilbert $A$-module together with a left action of $A$ on $X$. In \cite{Pi},
M.\thinspace V.\thinspace Pimsner first showed how to associate a $C^{\ast }$%
-algebra to certain $C^{\ast }$-correspondences, introducing a class of $%
C^{\ast }$-algebras that are now known as Cuntz-Pimsner algebras. It was
later that T.\thinspace Katsura, in his series of papers \cite{K1,K2,K3},
extended the former construction and associated a certain $C^{\ast }$%
-algebra to every $C^{\ast }$-correspondence. Katsura's more general
construction includes a wide range of algebras, amongst them the crossed
product of a $C^{\ast }$-algebra by a Hilbert $C^{\ast }$-bimodule, which
was introduced in \cite{AEE}.

The extension of so rich in results concepts to the case of pro-$C^{\ast }$%
-algebras could not be disregarded. A pro-$C^{\ast }$-algebra $A[\tau
_{\Gamma }]$ is a complete topological $\ast $-algebra for which there
exists a directed family of $C^{\ast }$-seminorms $\Gamma =\{p_{\lambda
}:\lambda \in \Lambda \}$ defining the topology $\tau _{\Gamma }$. In 1988,
N.C. Phillips considered Hilbert modules over pro-$C^{\ast }$-algebras and
studied their structure, in \cite{Ph}. An extensive survey of the theory of
Hilbert modules over pro-$C^{\ast }$-algebras can be found in \cite{J1}. In 
\cite{Z} the notion of a Hilbert pro-$C^{\ast }$-bimodule over a pro-$%
C^{\ast }$-algebra was defined. Subsequently, in \cite{JZ} we defined and
studied the crossed product of a pro-$C^{\ast }$-algebra by a Hilbert pro-$%
C^{\ast }$-bimodule, which is a generalization of crossed products of pro-$%
C^{\ast }$-algebras by inverse limit automorphisms (for the latter see \cite%
{J2}). All the above, gave us the impetus to generalize the important topic
of $C^{\ast }$-correspondences in the setting of pro-$C^{\ast }$-algebras
and to examine under which conditions we can associate a pro-$C^{\ast }$%
-algebra to a pro-$C^{\ast}$-correspondence (for the latter see Definition
3.1).

The paper is organized as follows. In Section 2 we gather some basic facts
on pro-$C^{*}$-algebras and Hilbert pro-$C^{*}$-modules that are needed for
understanding the main results of this paper. Sections 3 and 4 are devoted
in the definition of pro-$C^{*}$-correspondences and representations of them
respectively. In Section 5 we prove that for a certain pro-$C^{*}$%
-correspondence, namely an inverse limit pro-$C^{*}$-correspondence as we
shall call it, a universal pro-$C^{*}$-algebra can be associated to it, and
in Section 6, we see that in case $X$ is a Hilbert pro-$C^{*}$-bimodule over
a pro-$C^{*}$-algebra $A$ the crossed product of $A$ by $X$ is isomorphic to
the pro-$C^{*}$-algebra associated to $X$, when the latter is regarded as a
pro-$C^{*}$-correspondence. Finally, in Section 7, as an application, we
show how the association of a pro-$C^{*}$-algebra to a pro-$C^{*}$%
-correspondence described in Section 5, generalizes the construction of the
crossed product of a pro-$C^{*}$-algebra by a strong bounded automorphism.

\section{Preliminaries}

All vector spaces and algebras we deal with are considered over the field $%
\mathbb{C}$ of complex numbers and all topological spaces are assumed
Hausdorff.

A \textit{pro-}$C^{\ast }$\textit{algebra} $A[\tau _{\Gamma }]$ is a
complete topological $\ast $-algebra for which there exists an upward
directed family $\Gamma $ of $C^{\ast }$-seminorms $\{p_{\lambda
}\}_{\lambda \in \Lambda }$ defining the topology $\tau _{\Gamma }$ (\cite[%
Definition 7.5]{F}). Other terms with which pro-$C^{\ast }$-algebras can be
found in the literature are: locally $C^{\ast }$-algebras (A.\thinspace
Inoue), $b^{\ast }$-algebras (C.\thinspace Apostol) and LMC$^{\ast }$%
-algebras (G.\thinspace Lassner, K.\thinspace Schm{\"{u}}dgen).

For a pro-$C^{\ast }$-algebra $A[\tau _{\Gamma }]$ and for every $\lambda
\in \Lambda $, the quotient normed $\ast $\thinspace -algebra $A_{\lambda
}=A/N_{\lambda }$, where $N_{\lambda }=\{a\in A:\,p_{\lambda }(a)=0\}$, is
already complete, hence a $C^{\ast }$-algebra in the norm $||a+N_{\lambda
}||_{A_{\lambda }}=p_{\lambda }(a),\,a\in A$ (\cite[Theorem 10.24]{F}). The
canonical map from $A$ to $A_{\lambda }$ is denoted by $\pi _{\lambda }^{A}$%
. For $\lambda ,\,\mu \in \Lambda $ with $\lambda \geq \mu $, there is a
canonical surjective $C^{\ast }$-morphism $\pi _{\lambda \mu
}^{A}:A_{\lambda }\rightarrow A_{\mu }$, such that $\pi _{\lambda \mu
}^{A}(a+N_{\lambda })=a+N_{\mu }$ for all $a\in A$. The Arens-Michael
decomposition gives us the representation of $A$ as an inverse limit of $%
C^{\ast }$-algebras, namely $A=\lim\limits_{\leftarrow \lambda }A_{\lambda }$%
, up to a topological $\ast $-isomorphism (\cite[p.\,15-16]{F}). We refer
the reader to \cite{F} for further information about pro-$C^{\ast }$%
-algebras.

Given two pro-$C^{\ast }$-algebras $A[\tau _{\Gamma }]$ and $B[\tau _{\Gamma
^{\prime }}]$, a continuous $\ast $-morphism $\varphi :A\rightarrow B$ is
called a \textit{pro-$C^{\ast }$-morphism}.

Here we recall some basic facts from \cite{J1} and \cite{Z} regarding
Hilbert pro-$C^{*}$-modules and Hilbert pro-$C^{*}$-bimodules, respectively.

Let $A[\tau _{\Gamma }]$ be a pro-$C^{\ast }$-algebra. A \textit{right
Hilbert pro-}$C^{\ast }$\textit{-module} $X$ \textit{over }$A$\textit{\ }(or
just \textit{Hilbert }$A$\textit{-module}), is a linear space $X$ that is
also a right $A$-module equipped with a right $A$-valued inner product $%
\left\langle \cdot ,\cdot \right\rangle _{A}$, that is $\mathbb{C}$- and $A$%
-linear in the second variable and conjugate linear in the first variable,
with the following properties:

\begin{enumerate}
\item $\left\langle x,x\right\rangle _{A}\geq 0,\,\forall\,x\in X,$ and $%
\left\langle x,x\right\rangle _{A}=0$ if and only if $x=0$,

\item $\left( \left\langle x,y\right\rangle _{A}\right) ^{\ast
}=\left\langle y,x\right\rangle _{A},\,\forall\,x,y\in X,$
\end{enumerate}

and which is complete with respect to the topology given by the family of
seminorms $\{p_{\lambda }^{A}\}_{\lambda \in \Lambda },$ with $p_{\lambda
}^{A}\left( x\right) =p_{\lambda }\left( \left\langle x,x\right\rangle
_{A}\right) ^{\frac{1}{2}},x\in X$.

A Hilbert $A$-module is \emph{full} if the pro-$C^{\ast }$- subalgebra of $A$
generated by $\{\left\langle x,y\right\rangle _{A};x,y\in X\}$ coincides
with $A$.

A \textit{left Hilbert pro-}$C^{\ast }$\textit{-module }$X$\ over\ a\ pro-$%
C^{\ast }$-algebra $A[\tau _{\Gamma }]$ is defined in the same way, where
for instance the completeness is requested with respect to the family of
seminorms $\{^{A}p_{\lambda }\}_{\lambda \in \Lambda },$ where $%
^{A}p_{\lambda }\left( x\right) =p_{\lambda }\left( _{A}\left\langle
x,x\right\rangle \right) ^{\frac{1}{2}},x\in X$.

In case $X$ is a left Hilbert \text{pro-}$C^{\ast }$\text{-}module over $%
A[\tau _{\Gamma }]$ and a right Hilbert \text{pro-}$C^{\ast }$\text{-}module
over $B[\tau _{\Gamma ^{\prime }}]$ ($\tau _{\Gamma ^{\prime }}$ is given by
the family of $C^{\ast }$-seminorms $\{q_{\lambda }\}_{\lambda \in \Lambda }$%
), such that the following relations hold:

\begin{itemize}
\item $_{A}\left\langle x,y\right\rangle z=x\left\langle y,z\right\rangle
_{B}$ for all $x,y,z\in X$ ,

\item $q_{\lambda }^{B}(ax)$ $\leq p_{\lambda }(a)q_{\lambda }^{B}\left(
x\right) $ and $^{A}p_{\lambda }(xb)$ $\leq q_{\lambda }(b)^{A}p_{\lambda
}\left( x\right) $, for all $x\in X,\,a\in A,\,b\in B$ and for all $\lambda
\in \Lambda $,
\end{itemize}

then we say that $X$ is \textit{a Hilbert }$A-B$\textit{\ pro-}$C^{\ast }$%
\textit{-bimodule}.

A Hilbert $A-B$ pro-$C^{\ast }$\text{-}bimodule $X$ is \textit{full} if it
is full as a right and as a left Hilbert pro-$C^{*}$-module.

Let $\Lambda $ be an upward directed set and $\{A_{\lambda };B_{\lambda
};X_{\lambda };\pi _{\lambda \mu };\chi _{\lambda \mu };\sigma _{\lambda \mu
};\lambda ,\mu \in \Lambda ,\lambda \geq \mu \}$ an inverse system of
Hilbert $C^{\ast }$-bimodules, that is:

\begin{itemize}
\item $\{A_{\lambda };\pi _{\lambda \mu };\lambda ,\mu \in \Lambda ,\lambda
\geq \mu \}$ and $\{B_{\lambda };\chi _{\lambda \mu };\lambda ,\mu \in
\Lambda ,\lambda \geq \mu \}$ are inverse systems of $C^{\ast }$-algebras;

\item $\{X_{\lambda };\sigma _{\lambda \mu };\lambda ,\mu \in \Lambda
,\lambda \geq \mu \}$ is an inverse system of Banach spaces;

\item for each $\lambda \in \Lambda ,$ $X_{\lambda }$ is a Hilbert $%
A_{\lambda }-B_{\lambda }$ $C^{\ast }$-bimodule;

\item $\left\langle \sigma _{\lambda \mu }\left( x\right) ,\sigma _{\lambda
\mu }\left( y\right) \right\rangle _{B_{\mu }}=\chi _{\lambda \mu }\left(
\left\langle x,y\right\rangle _{B_{\lambda }}\right) $ and $_{A_{\mu
}}\left\langle \sigma _{\lambda \mu }\left( x\right) ,\sigma _{\lambda \mu
}\left( y\right) \right\rangle =\pi _{\lambda \mu }\left( _{A_{\lambda
}}\left\langle x,y\right\rangle \right) $, for all $x,y\in X_{\lambda }$ and
for all $\lambda ,\mu \in \Lambda $ with $\lambda \geq \mu ;$

\item $\sigma_{\lambda\mu}(x)\chi_{\lambda\mu}(b)=\sigma_{\lambda\mu}(xb),\,%
\,\pi_{\lambda\mu}(a)\sigma_{\lambda\mu}(x)=\sigma_{\lambda\mu}(ax)$, for
all $x\in X_{\lambda},\,a\in\,A_{\lambda},\,b\in\,B_{\lambda}$ and for all $%
\lambda,\mu\in\Lambda$ such that $\lambda\geq\mu.$
\end{itemize}

Let $A=${{$\lim\limits_{\leftarrow \lambda }A_{\lambda }$, }}$B=$$%
\lim\limits_{\leftarrow \lambda }B_{\lambda }${\ and }$X=${{$%
\lim\limits_{\leftarrow \lambda }X_{\lambda }$. Then }}$X$ has the structure
of a Hilbert $A-B$ pro-$C^{\ast }$-bimodule with

\begin{center}
$\left( x_{\lambda }\right) _{\lambda \in \Lambda }\left( b_{\lambda
}\right) _{\lambda \in \Lambda }=\left( x_{\lambda }b_{\lambda }\right)
_{\lambda \in \Lambda }$ and $\left\langle \left( x_{\lambda }\right)
_{\lambda \in \Lambda },\left( y_{\lambda }\right) _{\lambda \in \Lambda
}\right\rangle _{B}$ $=\left( \left\langle x_{\lambda },y_{\lambda
}\right\rangle _{B_{\lambda }}\right) _{\lambda \in \Lambda }$

and

$\left( a_{\lambda }\right) _{\lambda \in \Lambda }\left( x_{\lambda
}\right) _{\lambda \in \Lambda }=\left( a_{\lambda }x_{\lambda }\right)
_{\lambda \in \Lambda }$ and $_{A}\left\langle \left( x_{\lambda }\right)
_{\lambda \in \Lambda },\left( y_{\lambda }\right) _{\lambda \in \Lambda
}\right\rangle =\left( _{A_{\lambda }}\left\langle x_{\lambda },y_{\lambda
}\right\rangle \right) _{\lambda \in \Lambda }$,
\end{center}

where $(x_{\lambda})_{\lambda\in\Lambda}\in
X,\,(b_{\lambda})_{\lambda\in\Lambda}\in B$ and $(a_{\lambda})_{\lambda\in%
\Lambda}\in A.$

Let $X$ be a Hilbert $A-B$ pro-$C^{\ast }$-bimodule. Then, for each $\lambda
\in \Lambda ,$ $^{A}p_{\lambda }\left( x\right) =q_{\lambda }^{B}\left(
x\right)$, for all $x\in X$, and the normed space $X_{\lambda }=X/N_{\lambda
}^{B}$, where $N_{\lambda }^{B}=\{x\in X;q_{\lambda }^{B}\left( x\right)
=0\} $, is complete in the norm $||x+N_{\lambda }^{B}||_{X_{\lambda }}$ $%
=q_{\lambda }^{B}(x),x\in X$. Moreover, $X_{\lambda }$ has a canonical
structure of a Hilbert $A_{\lambda }-$ $B_{\lambda }$ $C^{\ast }$-bimodule
with $\left\langle x+N_{\lambda }^{B},y+N_{\lambda }^{B}\right\rangle
_{B_{\lambda }}$ $=\left\langle x,y\right\rangle _{B}+\ker q_{\lambda }$ and 
$_{A_{\lambda }}\left\langle x+N_{\lambda }^{B},y+N_{\lambda
}^{B}\right\rangle =_{A}\left\langle x,y\right\rangle +\ker p_{\lambda }$,
for all $x,y\in X$. The canonical surjection from $X$ on $X_{\lambda }$ is
denoted by $\sigma _{\lambda }^{X}$. For $\lambda ,\mu \in \Lambda $ with $%
\lambda \geq \mu $, there is a canonical surjective linear map $\sigma
_{\lambda \mu }^{X}:X_{\lambda }\rightarrow X_{\mu }$ such that $\sigma
_{\lambda \mu }^{X}\left( x+N_{\lambda }^{B}\right) =x+N_{\mu }^{B}$ for all 
$x\in X$. Then $\{A_{\lambda };B_{\lambda };X_{\lambda };\pi _{\lambda \mu
}^{A};\sigma _{\lambda \mu }^{X};\pi _{\lambda \mu }^{B};\lambda ,\mu \in
\Lambda ,\lambda \geq \mu \}$ is an inverse system of Hilbert $C^{\ast }$%
-bimodules in the above sense.

Let $X$ be a Hilbert pro-$C^{\ast }$-module over $B$. A morphism $%
T:X\rightarrow X$ of right modules is \textit{adjointable} if there is
another morphism of modules $T^{\ast }:X\rightarrow X\ $such that $%
\left\langle Tx_{1},x_{2}\right\rangle _{B}=\left\langle x_{1},T^{\ast
}x_{2}\right\rangle _{B}$ for all $x_{1},x_{2}\in X$. The vector space $%
L_{B}(X)\ $of all adjointable module morphisms from $X$ to $X$ has a
structure of a pro-$C^{\ast }$-algebra under the topology given by the
family of $C^{\ast }$-seminorms $\{q_{\lambda ,L_{B}(X)}\}_{\lambda \in
\Lambda }$, where 
\begin{equation*}
q_{\lambda ,L_{B}(X)}\left( T\right) =\sup \{q_{\lambda }^{B}(Tx):q_{\lambda
}^{B}\left( x\right) \leq 1\},\,\forall\,\lambda\in\Lambda,\,T\in L_{B}(X).
\end{equation*}
Moreover, $\{L_{B_{\lambda }}(X_{\lambda });\pi _{\lambda \mu }^{L_{B}(X)}$; 
$\lambda ,\mu \in \Lambda ,\lambda \geq \mu \}$ where $\pi _{\lambda \mu
}^{L_{B}(X)}:L_{B_{\lambda }}(X_{\lambda })\rightarrow L_{B_{\mu }}(X_{\mu
}) $ is given by $\pi _{\lambda \mu }^{L_{B}(X)}\left( T\right) \left(
\sigma _{\mu }^{X}\left( x\right) \right) =\sigma _{\lambda \mu }^{X}\left(
T(\sigma _{\lambda }^{X}\left( x\right) )\right) $, for all $T\in
L_{B}(X),\,x\in X$, is an inverse system of $C^{\ast }$-algebras and $%
L_{B}(X)=${{$\lim\limits_{\leftarrow \lambda }$}} $L_{B_{\lambda
}}(X_{\lambda })$, up to an isomorphism of pro-$C^{\ast }$-algebras. The
canonical projections $\pi _{\lambda }^{L_{B}(X)}:$ $L_{B}(X)\rightarrow $ $%
L_{B_{\lambda }}(X_{\lambda }),$ $\lambda \in \Lambda $, are given by $\pi
_{\lambda }^{L_{B}(X)}\left( T\right) \left( \sigma _{\lambda }^{X}\left(
x\right) \right) =\sigma _{\lambda }^{X}\left( T(x)\right) $ for all $T\in\L %
_{B}(X)$ and $x\in X$. For $x,y\in X$, the map 
\begin{equation*}
\theta _{y,x}:X\rightarrow X, \,\text{given by}\,\,\, \theta _{y,x}\left(
z\right) =y\left\langle x,z\right\rangle _{B},\,\forall\,x,y,z\in X,
\end{equation*}
is an adjointable module morphism. $\Theta \left( X\right):=$span$\{\theta
_{y,x}:x,y\in X\}$ is a two-sided $\ast $-ideal of $L_{B}(X)$ and its
closure in $L_{B}(X)$ is denoted by $K_{B}(X)$. Moreover, $\left(
K_{B}(X)\right) _{\lambda }=K_{B_{\lambda }}(X_{\lambda })$, for each $%
\lambda \in \Lambda$, with respect to an isomorphism of C$^{*}$-algebras.

Throughout this paper, $A$ and $B$ are pro-$C^{\ast }$-algebras whose
topologies are given by the families of $C^{\ast }$-seminorms $\{p_{\lambda
},\lambda \in \Lambda \}$, respectively $\{q_{\delta },\delta \in \Delta \}.$

\section{Pro-$C^{\ast }$-correspondences}

\begin{definition}
A pro-$C^{\ast }$-correspondence \emph{is a triple $\left( X,A,\varphi
_{X}\right) $, where $A$ is a pro-$C^{\ast }$-algebra, $X$ is a Hilbert pro-$%
C^{\ast }$-module over $A$ and $\varphi _{X}:A\rightarrow L_{A}(X)$ is a pro-%
$C^{\ast }$-morphism.}

\emph{\ A pro-$C^{\ast }$-correspondence $\left( X,A,\varphi _{X}\right) $}
is nondegenerate \emph{\ if $\varphi _{X}$ is nondegenerate (that is, $\left[
\varphi _{X}(A)X\right] =X$, where $[\varphi _{X}(A)X]$ stands for the
closure of the linear span of the set $\{\varphi _{X}(a)x:a\in A,\,x\in X\}$
).}
\end{definition}

\begin{example}
\label{automorphism}\emph{Let $A\ $be a pro-$C^{\ast }$-algebra and $\alpha
:A\rightarrow A$ a nondegenerate pro-$C^{\ast }$-morphism. Consider $\varphi
_{A}:A\rightarrow L_{A}(A)\ $defined by $\varphi _{A}\left( a\right) \left(
b\right) =\alpha \left( a\right) b,\,a,b\in A$. Clearly, $\varphi _{A}$ is a
pro-$C^{\ast }$-morphism and $\left[ \varphi _{A}(A)A\right] =A$. Therefore, 
$\left( A,A,\varphi _{A}\right) $ is a nondegenerate pro-$C^{\ast }$%
-correspondence. If $\alpha =$id$_{A}$, we say that $\left( A,A,\text{id}%
_{A}\right) $ } is the identity pro-$C^{\ast }$-correspondence.
\end{example}

\begin{example}
\emph{Suppose that $X$ is a Hilbert $A-A$ pro-$C^{\ast }$-bimodule. Then the
map $\varphi _{X}:A\rightarrow L_{A}(X)$ defined by $\varphi _{X}\left(
a\right) \left( x\right) =ax$,$\,a\in A,\,x\in X$, is a pro-$C^{\ast }$%
-morphism and since $\left[ AX\right] =X,$ $\left( X,A,\varphi _{X}\right) $
is a nondegenerate pro-$C^{\ast }$-correspondence.}
\end{example}

\begin{example}
\emph{Suppose that $\left( X,A,\varphi _{X}\right) $ and $\left( Y,A,\varphi
_{Y}\right)\ $ are pro-$C^{\ast }$-correspondences. By \cite[pp.77-79]{J1}, $%
X\otimes _{\varphi _{Y}}Y\ $is a Hilbert pro-$C^{\ast }$-module over $A\ $%
and the map $\varphi _{X\otimes _{\varphi _{Y}}Y}:A\rightarrow
L_{A}(X\otimes _{\varphi _{Y}}Y)$ defined by 
\begin{equation*}
\varphi _{X\otimes _{\varphi _{Y}}Y}\left( a\right) \left( x\otimes
_{\varphi _{Y}}y\right) =\varphi _{X}\left( a\right) \left( x\right) \otimes
_{\varphi _{Y}}y,\,a\in A,\,x\in X,\,y\in Y,
\end{equation*}%
is a pro-$C^{\ast }$-morphism \cite[Proposition 4.3.4]{J1}. Then $\left(
X\otimes _{\varphi _{Y}}Y,A,\varphi _{X\otimes _{\varphi _{Y}}Y}\right) $ is
a pro-$C^{\ast }$-correspondence called the tensor product of the pro-$%
C^{\ast }$-correspondences $\left( X,A,\varphi _{X}\right) $ and $\left(
Y,A,\varphi _{Y}\right) $.}
\end{example}

\begin{definition}
\emph{A pro-$C^{\ast }$-correspondence $\left( X,A,\varphi _{X}\right) $} is
an inverse limit pro-$C^{\ast }$-correspondence,\emph{\ if $A$ is an inverse
limit, $\lim\limits_{\leftarrow \lambda }A_{\lambda }$, of $C^{\ast }$%
-algebras in such a way that $X$ is an inverse limit, $\lim\limits_{%
\leftarrow \lambda }X_{\lambda }$, of Hilbert $C^{\ast }$-modules, where $%
X_{\lambda }$ is a Hilbert $A_{\lambda }$-module for each $\lambda $ and $%
\varphi _{X}$ is an inverse limit, $\lim\limits_{\leftarrow \lambda }\varphi
_{X_{\lambda }}$, of $C^{\ast }$-morphisms.}
\end{definition}

\begin{example}
\emph{The identity pro-$C^{\ast }$-correspondence and the Hilbert pro-$%
C^{\ast }$-bimodules are inverse limit pro-$C^{\ast }$-correspondences.}
\end{example}

Throughout this paper an ideal of a pro-$C^{\ast }$-algebra always means a
closed two-sided $\ast $- ideal.\textbf{\ } For a pro-$C^{\ast }$%
-correspondence $\left( X,A,\varphi _{X}\right) $ and an ideal $I$ of $A$,
the following ideals of $A$ are defined (see \cite[Definition 4.1]{K3}): 
\begin{eqnarray*}
X(I) &=&\overline{\text{span}}\{<y,\varphi _{X}(a)x>_{A}\in A:\,a\in
I,\,x,y\in X\}, \\
X^{-1}(I) &=&\{a\in \,A:\,<y,\varphi _{X}(a)x>_{A}\in \,I,\, \forall
\,x,y\in \,X\}\text{.}
\end{eqnarray*}

\begin{lemma}
\label{simple lemma} Let $X$ be a Hilbert $A$-module and $I$ an ideal of $A$%
. We put $XI=$span$\{xa:x\in X,a\in I\}$. Then $x\in XI$ if and only if $%
<y,x>_{A}\in I$, for all $y\in X$.
\end{lemma}

\begin{proof}
The forward implication is immediate. For the inverse, we have that $%
<x,x>_{A}\in I$, hence from \cite[Corollary 1.3.11]{J1}, if $\alpha $\ is a
real number, $0<\alpha<\frac{1}{2}$, then there exists $y\in X$, such that $%
x=y<x,x>_{A}^{\alpha }$. From functional calculus in pro-$C^{\ast }$%
-algebras (see \cite{F}), we then have that $<x,x>_{A}^{a}\in I$, so $x\in
XI.$\
\end{proof}

\noindent Based on the previous lemma, we get that $XI$\ is a closed
submodule of $X$. In particular if $I=\ker p_{\lambda }$, then by a proof
similar to that of Lemma \ref{simple lemma}, we have that $\ker p_{\lambda
}^{A}=X\ker p_{\lambda }$, so $X/X\ker p_{\lambda }=X_{\lambda }.$

\begin{remark}
\label{inv}\emph{If by $\phi _{I}$\ we denote the $\ast $-morphism $\phi
_{I}:L_{A}(X)\rightarrow L_{A}(X/XI)$ given by $:$%
\begin{equation*}
\phi _{I}(T)(x+XI)=Tx+XI,\,T\in L_{A}(X),\,x\in X,
\end{equation*}%
\ then we get that $X^{-1}(I)=\ker (\phi _{I}\circ \varphi _{X})$. In
particular, if $I=\ker p_{\lambda }$, then $X^{-1}(\ker p_{\lambda })=\ker
(\pi _{\lambda }^{L_{A}(X)}\circ \varphi _{X}).$}
\end{remark}

\begin{lemma}
\label{steplemma1} A pro-$C^{\ast }$-correspondence $\left( X,A,\varphi
_{X}\right) $ is an inverse limit pro-$C^{\ast }$-correspondence if and only
if $X(\ker p_{\lambda })\subset \ker p_{\lambda }\ $, for all $\lambda \in
\Lambda $.
\end{lemma}

\begin{proof}
Suppose that $\left( X,A,\varphi _{X}\right) $ is an inverse limit pro-$%
C^{\ast }$-correspondence. Then $\varphi
_{X}=\lim\limits_{\leftarrow \lambda }\varphi _{X_{\lambda }}$. Let
$<y,\varphi _{X}(a)x>_{A}\in X(\ker p_{\lambda })$, for $x,y\in
X,\,a\in \ker p_{\lambda}$. Then
\begin{eqnarray*}
\pi _{\lambda }^{A}\left( <y,\varphi _{X}(a)x>_{A}\right) &=&<\sigma
_{\lambda }^{X}\left( y\right) ,\pi _{\lambda }^{L_{A}\left( X\right)
}\left( \varphi _{X}(a)\right) \sigma _{\lambda }^{X}\left( x\right)
>_{A_{\lambda }} \\
&=&<\sigma _{\lambda }^{X}\left( y\right) ,\varphi _{X_{\lambda }}(\pi
_{\lambda }^{A}\left( a \right) )\sigma _{\lambda }^{X}\left( x\right)
>_{A_{\lambda }}=0,
\end{eqnarray*}%
and so $<y,\varphi _{X}( a )x>_{A}\in \ker p_{\lambda }.$

Conversely, let $\lambda \in \Lambda $. If $a\in \ker p_{\lambda }$, then $%
<y,\varphi _{X}(a)x>_{A}\in X(\ker p_{\lambda })\subset \ker
p_{\lambda }$, for all $x,y\,\in X$, $\,$whence $\varphi _{X}(a)x\in
\ker p_{\lambda }^{A}$, for all $x\,\in X.$ Also, since $\ker
p_{\lambda }^{A}=X\ker p_{\lambda }$, as noted after Lemma
\ref{simple lemma}, the submodule $\ker
p_{\lambda }^{A}$ of $X$ remains invariant under the action of ${\varphi }${%
$_{X}(A)$.} Therefore, we can consider a linear map $\varphi _{X_{\lambda
}}:A_{\lambda }\rightarrow L_{A_{\lambda }}(X_{\lambda })$ defined by
\begin{equation*}
\,\varphi _{X_{\lambda }}(\pi _{\lambda }^{A}(a))(\sigma _{\lambda }^{X}(x))%
\newline
=\sigma _{\lambda }^{X}(\varphi _{X}(a)x),\,\forall a\in A,\,x\in
X\text{.}
\end{equation*}%
It is easy to check that $\left( \varphi _{X_{\lambda }}\right) _{\lambda }$
is an inverse system of $C^{\ast }$-morphisms, such that $\varphi _{X}={%
\lim\limits_{\leftarrow \lambda }}\varphi _{X_{\lambda }}$, and thus $\left(
X\emph{,}A,\varphi _{X}\right) $ is an inverse limit pro-$C^{\ast }$%
-correspondence.
\end{proof}

\section{Representations of pro-$C^{\ast }$-correspondences}

\begin{definition}
\label{morphism of correspondences} A morphism from a pro-$C^{\ast }$%
-correspondence \emph{$\left( X,A,\varphi _{X}\right) $} to a pro-$C^{\ast }$%
-correspondence \emph{$\left( Y,B,\varphi _{Y}\right) $ is a pair $\left(
\Pi ,T\right) $ consisting of a pro-$C^{\ast }$-morphism $\Pi :A\rightarrow
B $ and a map $T:X\rightarrow Y$ such that the following conditions are met:}

\begin{enumerate}
\item[\emph{(1)}] \emph{\ $\left\langle T(x_{1}),T(x_{2})\right\rangle
_{B}=\Pi \left( \left\langle x_{1},x_{2}\right\rangle _{A}\right) $, for all 
$x_{1},x_{2}\in A;$}

\item[\emph{(2)}] \emph{$\varphi _{Y}\left( \Pi \left( a\right) \right)
T\left( x\right) =T\left( \varphi _{X}\left( a\right) x\right) $, for all $%
a\in A$ and for all $x\in X.$}
\end{enumerate}

\emph{We say that} the morphism \emph{$\left( \Pi ,T\right) $} is
nondegenerate \emph{if $\left[ \Pi \left( A\right) B\right] =B$ and $\left[
T\left( X\right) B\right] =Y.$} \emph{\ }
\end{definition}

\begin{remark}
\emph{Let $\left( \Pi ,T\right) $ be a morphism from a pro-$C^{\ast }$%
-correspondence $\left( X,A,\varphi _{X}\right) $ to a pro-$C^{\ast }$%
-correspondence $\left( Y,B,\varphi _{Y}\right) $. Then:}

\begin{enumerate}
\item[\emph{(1)}] \emph{$T$ is a continuous linear map.}

\item[\emph{(2)}] \emph{$T\left( x\right) \Pi \left( a\right) =T\left(
xa\right) $, for all $a\in A,$ $x\in X.$}
\end{enumerate}
\end{remark}

\begin{proof}
$(1)$ A simple calculation, based on relation (1) of Definition \ref%
{morphism of correspondences}, shows that $T$ is linear.

For each $\delta \in \Delta $, there is $\lambda \in \Lambda $ such that
\begin{equation*}
q_{\delta }^{B}\left( T\left( x\right) \right) ^{2}=q_{\delta }\left( \Pi
\left( \left\langle x,x\right\rangle _{A}\right) \right) \leq p_{\lambda
}\left( \left\langle x,x\right\rangle _{A}\right) =p_{\lambda }^{A}\left(
x\right) ^{2}
\end{equation*}%
for all $x\in X.$

$(2)$ For each $\delta \in \Delta $, we have
\begin{eqnarray*}
&&q_{\delta }^{B}\left( T\left( x\right) \Pi \left( a\right) -T\left(
xa\right) \right) ^{2} \\
&=&q_{\delta }\left( \left\langle T\left( x\right) \Pi \left( a\right)
-T\left( xa\right) ,T\left( x\right) \Pi \left( a\right) -T\left( xa\right)
\right\rangle \right) \\
&=&q_{\delta }\left( \Pi \left( a^{\ast }\left\langle
x,x\right\rangle _{A}a\right) -\Pi \left( a^{\ast }\left\langle
x,xa\right\rangle _{A}\right) -\Pi \left( \left\langle
xa,x\right\rangle _{A}a\right) +\Pi \left( \left\langle
xa,xa\right\rangle _{A}\right) \right) =0,
\end{eqnarray*}%
for all $a\in A,$ $x\in X$.
\end{proof}

For the proof of Lemma \ref{the map psi}, we use the following result from 
\cite{KPW}.

\begin{lemma}
\emph{(\cite[Lemma 2.2]{KPW})} If $A$ is a $C^{\ast }$-algebra and $X$ is a
Hilbert $A$-module, then for $n\in\mathbb{N}$ and $x_{1},\cdots
,x_{n},y_{1},\cdots ,y_{n}\in X$ we get that 
\begin{equation*}
||\sum_{i=1}^{n}\theta _{x_{i},y_{i}}||=||(\left[ <x_{i},x_{j}>_{A}\right]
_{i,j=1}^{n})^{\frac{1}{2}}(\left[ <y_{i},y_{j}>_{A}\right] _{i,j=1}^{n})^{%
\frac{1}{2}}||,
\end{equation*}%
where the norm in the right hand side is the norm in the $C^{\ast }$-algebra 
$M_{n}(A)$, of all $n\times n$ matrices with entries from $A$.
\end{lemma}

\begin{lemma}
\label{the map psi} For a representation $(\Pi ,T)$ from a pro-$C^{\ast }$%
-correspondence $\left( X,A,\varphi _{X}\right) $ to a pro-$C^{\ast }$%
-correspondence $\left( Y,B,\varphi _{Y}\right) $, there is a pro-$C^{\ast }$%
-morphism $\psi _{T}:K_{A}(X)\rightarrow K_{B}(Y)$, such that $\psi
_{T}(\theta _{x,y})=\theta _{T(x),T(y)},$ for all $x,y\in X$.
\end{lemma}

\begin{proof}
It suffices to show that $\psi _{T}|_{\Theta (X)}$\ is continuous. Since $%
\Pi $ is continuous, for each $\delta \in \Delta $, there is
$\lambda \in \Lambda $, such that $q_{\delta }\left( \Pi \left(
a\right) \right) \leq
p_{\lambda }(a)$, for all $a\in A$, and so there is a $C^{\ast }$-morphism $%
\Pi _{\delta }:A_{\lambda }\rightarrow B_{\delta } $ such that $\pi
_{\delta }^{B}\circ \Pi =\Pi _{\delta }\circ \pi _{\lambda }^{A}$.
Then for each $ \delta \in \Delta $, we have
\begin{eqnarray*}
&&q_{\delta ,L_{B}(Y)}(\psi _{T}(\sum_{j=1}^{n}\theta _{x_{j},y_{j}})) \\
&=&q_{\delta ,L_{B}(Y)}(\sum_{j=1}^{n}\theta
_{T(x_{j}),T(y_{j})})=||\sum_{j=1}^{n}\theta _{\sigma _{\delta }^{Y}\left(
T(x_{j})\right) ,\sigma _{\delta }^{Y}\left( T(y_{j})\right) }|| \\
&=&||(\left[ \pi _{\delta }^{B}(\left\langle T(x_{i}),T(x_{j})\right\rangle
_{B})\right] _{i,j=1}^{n})^{\frac{1}{2}}(\left[ \pi _{\delta
}^{B}(\left\langle T(y_{i}),T(y_{j})\right\rangle _{B})\right]
_{i,j=1}^{n})^{\frac{1}{2}}|| \\
&=&||(\left[ \pi _{\delta }^{B}\circ \Pi (\left\langle
x_{i},x_{j}\right\rangle _{A})\right] _{i,j=1}^{n})^{\frac{1}{2}}(\left[ \pi
_{\delta }^{B}\circ \Pi (\left\langle y_{i},y_{j}\right\rangle _{A})\right]
_{i,j=1}^{n})^{\frac{1}{2}}|| \\
&=&||(\left[ \Pi _{\delta }\circ \pi _{\lambda }^{A}(\left\langle
x_{i},x_{j}\right\rangle _{A})\right] _{i,j=1}^{n})^{\frac{1}{2}}(\left[ \Pi
_{\delta }\circ \pi _{\lambda }^{A}(\left\langle y_{i},y_{j}\right\rangle
_{A})\right] _{i,j=1}^{n})^{\frac{1}{2}}|| \\
&=&||(\left[ \Pi _{\delta }(\left\langle \sigma _{\lambda }^{X}\left(
x_{i}\right) ,\sigma _{\lambda }^{X}\left( x_{j}\right) \right\rangle _{A})%
\right] _{i,j=1}^{n})^{\frac{1}{2}}(\left[ \Pi _{\delta }(\left\langle
\sigma _{\lambda }^{X}\left( y_{i}\right) ,\sigma _{\lambda }^{X}\left(
y_{j}\right) \right\rangle _{A})\right] _{i,j=1}^{n})^{\frac{1}{2}}|| \\
&\leq &||(\left[ \left\langle \sigma _{\lambda }^{X}\left( x_{i}\right)
,\sigma _{\lambda }^{X}\left( x_{j}\right) \right\rangle \right]
_{i,j=1}^{n})^{\frac{1}{2}}(\left[ \left\langle \sigma _{\lambda }^{X}\left(
y_{i}\right) ,\sigma _{\lambda }^{X}\left( y_{j}\right) \right\rangle \right]
_{i,j=1}^{n})^{\frac{1}{2}}|| \\
&=&||\sum_{j=1}^{n}\theta _{\sigma _{\lambda }^{X}\left(
x_{j}\right) ,\sigma _{\lambda }^{X}\left( y_{j}\right)
}||=p_{\lambda ,L_{A}\left( X\right) }\left( \sum_{j=1}^{n}\theta
_{x_{j},y_{j}}\right),
\end{eqnarray*}%
for all $x_{1},...,x_{n},y_{1},...,y_{n}\in X,\,n\in\mathbb{N}.$
\end{proof}

\noindent Let $\left( X,A,\varphi _{X}\right) $ be a pro-$C^{\ast }$%
-correspondence. For each $\lambda \in \Lambda $, we define the ideals%
\newline
\begin{equation*}
J_{X}^{\lambda }=\{a\in A:\pi _{\lambda }^{L_{A}(X)}\left( \varphi
_{X}\left( a\right) \right) \in K_{A_{\lambda }}(X_{\lambda })\text{ and }%
\pi _{\lambda }^{A}\left( ab\right) =0,\, \forall \,b\in \ker (\pi _{\lambda
}^{L_{A}(X)}\circ \varphi _{X})\}
\end{equation*}%
and 
\begin{equation*}
\mathcal{J}_{X}=\bigcap_{\lambda }J_{X}^{\lambda }.
\end{equation*}

\begin{remark}
\emph{For a $C^{\ast }$-correspondence $\left( X,A,\varphi _{X}\right) ,$ $%
J_{X}=\varphi _{X}^{\emph{-}1}(K_{A}(X))\cap (\ker \varphi _{X})^{\bot }$ 
\cite[Definition 3.3]{K3} is the largest ideal to which the restriction of $%
\varphi _{X}$ is an injection into $K_{A}(X)$. If $\left( X,A,\varphi
_{X}\right) $ is a $C^{\ast }$-correspondence, then}%
\begin{eqnarray*}
\mathcal{J}_{X} &=&\{a\in A:\varphi _{X}\left( a\right) \in K_{A}(X)\text{ 
\emph{and} }ab=0,\,\forall\, b\in \ker \varphi _{X}\} \\
&=&\varphi _{X}^{-1}(K_{A}(X))\cap (\ker \varphi _{X})^{\bot }=J_{X}.
\end{eqnarray*}
\end{remark}

\begin{lemma}
\label{ideal}Let $\left( X,A,\varphi _{X}\right) $ be an inverse limit pro-$%
C^{*}$-correspondence. Then $\pi _{\lambda }^{A}\left( J_{X}^{\lambda
}\right) =J_{X_{\lambda }}$ for all $\lambda \in \Lambda .$
\end{lemma}

\begin{proof}
If $\left( X,A,\varphi _{X}\right) $ is an inverse limit
correspondence, then $\varphi _{X}=$ $\lim\limits_{\leftarrow
\lambda }\varphi _{X_{\lambda }}$ and $\pi _{\lambda }^{L_{A}\left(
X\right) }\circ \varphi _{X}=\varphi _{X_{\lambda }}\circ \pi
_{\lambda }^{A}$, for all $\lambda \in \Lambda .$
Therefore,%
\begin{eqnarray*}
&&\pi _{\lambda }^{A}\left( J_{X}^{\lambda }\right) \\
&=&\{\pi _{\lambda }^{A}\left( a\right) \in A_{\lambda }:\varphi
_{X_{\lambda }}\left( \pi _{\lambda }^{A}\left( a\right) \right) \in
K_{A_{\lambda }}(X_{\lambda }),\text{ }\pi _{\lambda }^{A}\left(
a\right) \pi _{\lambda }^{A}\left( b\right) =0,\,\forall \,b\in \ker
(\varphi
_{X_{\lambda }}\circ \pi _{\lambda }^{A})\} \\
&=&\{\pi _{\lambda }^{A}\left( a\right) \in A_{\lambda }:\varphi
_{X_{\lambda }}\left( \pi _{\lambda }^{A}\left( a\right) \right) \in
K_{A_{\lambda }}(X_{\lambda }),\text{ }\pi _{\lambda }^{A}\left(
a\right) \pi _{\lambda }^{A}\left( b\right) =0,\,\forall\, \pi
_{\lambda
}^{A}\left( b\right) \in \ker \varphi _{X_{\lambda }}\} \\
&=&J_{X_{\lambda }}
\end{eqnarray*}%
for all $\lambda \in \Lambda .$
\end{proof}

\begin{definition}

\begin{enumerate}
\item[\emph{(1)}] A representation of a pro-$C^{\ast }$-correspondence \emph{%
$\left( X,A,\varphi _{X}\right) $} \ on a pro-$C^{\ast }$-algebra $B$\ \emph{%
is a morphism $(\pi ,t)$ from $\left( X,A,\varphi _{X}\right) $ to the
identity correspondence $\left( B,B,\text{id}_{B}\right) .$}

\item[\emph{(2)}] A covariant representation of a pro-$C^{\ast }$%
-correspondence \emph{$\left( X,A,\varphi _{X}\right) $} on a pro-$C^{\ast }$%
-algebra $B$ \emph{is a representation $\left( \pi ,t\right) $ with the
property that $\psi _{t}\left( \varphi _{X}\left( a\right) \right) =\pi
\left( a\right) $, for all $a\in \mathcal{J}_{X}.$}
\end{enumerate}
\end{definition}

Remark that in case $(\pi ,t)$ is a morphism of a pro-$C^{\ast }$%
-correspondence $\left( X,A,\varphi _{X}\right) $ on a pro-$C^{\ast }$%
-algebra $B$, then the map $\psi _{t}:K_{A}(X)\rightarrow B$ of Lemma \ref%
{the map psi} is given by $\psi _{t}(\theta _{x,y})=t(x)t(y)^{\ast },$ for $%
x,y\in X$. This is a consequence of Proposition \ref{generalresult} below
and the fact that every pro-$C^{\ast }$-algebra has an approximate identity
(see \cite{F}).

\section{Pro-$C^{\ast }$-algebras associated to pro-$C^{\ast }$%
-correspon\-dences}

For a representation $\left( \pi ,t\right) $ of a pro-$C^{\ast }$%
-correspondence $\left( X\emph{,}A\emph{,}\,\varphi _{X}\right) $ on a pro-$%
C^{\ast }$-algebra $B$, we denote by pro-$C^{\ast }$-$(\pi (A),t(X))$ the
pro-$C^{\ast }$-subalgebra of $B$ generated by the images of $\pi $ and $t$.

\begin{definition}
\label{universal} \emph{For a pro-$C^{\ast }$-correspondence $\left( X\emph{,%
}A\emph{,}\varphi _{X}\right) $,} the pro-$C^{\ast }$-algebra $\mathcal{O}%
_{X}$ is defined to be the pro-$C^{\ast }$-algebra\emph{\ pro-$C^{\ast }$-$%
(\pi _{X}(A),t_{X}(X))$, where $(\pi _{X},t_{X})$ is a universal covariant
representation of $X$, in the sense that for every covariant representation $%
(\pi ,t)$ of $X$ on a pro-$C^{\ast }$-algebra $B$, there is a unique pro-$%
C^{\ast }$-morphism $\Phi :\mathcal{O}_{X}\longrightarrow B$, such that $%
\Phi \circ \pi _{X}=\pi ,\,\Phi \circ t_{X}=t.$}
\end{definition}

\begin{remark}
\label{uniqueness}

\begin{enumerate}
\item[\emph{(1)}] \emph{If $\left( X\emph{,}A\emph{,}\varphi _{X}\right) $
is a $C^{\ast }$-correspondence, then } \emph{$\mathcal{O}_{X}$ is the $%
C^{\ast }$-algebra associated to it \cite[Definition 2.6]{K1}.}

\item[\emph{(2)}] \emph{Let $\left( X\emph{,}A\emph{,}\varphi _{X}\right) $
be a pro-$C^{\ast }$-correspondence. If the pro-$C^{\ast }$-algebra $%
\mathcal{O}_{X}$ exists, it is unique, up to a pro-$C^{\ast }$-isomorphism.}
\end{enumerate}
\end{remark}

\begin{lemma}
\label{steplemma} Let $\left( X\emph{,}A,\varphi _{X}\right) $ be an inverse
limit pro-$C^{\ast }$-correspondence with the property that $\pi _{\lambda
\mu }^{A}\left( J_{X_{\lambda }}\right) \subset J_{X_{\mu }}~$, for all $%
\lambda ,\mu \in \Lambda $ with $\lambda \geq \mu $. Then for each $\lambda
,\mu \in \Lambda $ with $\lambda \geq \mu $, there is a $C^{\ast }$-morphism 
$\rho _{\lambda \mu }:$ $\mathcal{O}_{X_{\lambda }}$ $\rightarrow \mathcal{O}%
_{X_{\mu }}$ such that $\rho _{\lambda \mu }\circ t_{X_{\lambda }}=t_{X_{\mu
}}\circ \sigma _{\lambda \mu }^{X}$ and $\rho _{\lambda \mu }\circ \pi
_{X_{\lambda }}=\pi _{X_{\mu }}\circ \pi _{\lambda \mu }^{A},$ where $(\pi
_{X_{\lambda }},t_{X_{\lambda }})$ is the universal covariant representation
of Definition \ref{universal}. Moreover, $\{\mathcal{O}_{X_{\lambda }};\rho
_{\lambda \mu };\lambda ,\mu \in \Lambda ,\lambda \geq \mu \}$ is an inverse
system of $C^{\ast }$-algebras.
\end{lemma}

\begin{proof}
We easily get that for all $\lambda \geq \mu ,$ the pair $(\pi _{X_{\mu
}}\circ \pi _{\lambda \mu }^{A},t_{X_{\mu }}\circ \sigma _{\lambda \mu
}^{X}) $ is a representation of the $C^{\ast }$-correspondence $\left(
X_{\lambda },A_{\lambda },\varphi _{X_{\lambda }}\right) $ on the $C^{\ast }$%
-algebra $\mathcal{O}_{X_{\mu }}$. We will show that this representation is
also a covariant representation. From
\begin{eqnarray*}
\psi _{t_{X_{\mu }}\circ \sigma _{\lambda \mu }^{X}}(\theta _{\sigma
_{\lambda }^{X}(x),\sigma _{\lambda }^{X}(y)}) &=&t_{X_{\mu }}\circ \sigma
_{\lambda \mu }^{X}\left( \sigma _{\lambda }^{X}(x)\right) \left( t_{X_{\mu
}}\circ \sigma _{\lambda \mu }^{X}\left( \sigma _{\lambda }^{X}(y)\right)
\right) ^{\ast } \\
&=&t_{X_{\mu }}\left( \sigma _{\mu }^{X}(x)\right) t_{X_{\mu }}\left( \sigma
_{\mu }^{X}(y)\right) ^{\ast } \\
&=&\psi _{t_{X_{\mu }}}(\theta _{\sigma _{\mu }^{X}(x),\sigma _{\mu
}^{X}(y)})=\psi _{t_{X_{\mu }}}\left( \pi _{\lambda \mu
}^{L_{A}(X)}\left( \theta _{\sigma _{\lambda }^{X}(x),\sigma
_{\lambda }^{X}(y)}\right) \right),
\end{eqnarray*}%
for all $x,y\in X$, and taking into account that for all $\lambda\in\Lambda$%
, $\Theta \left( X_{\lambda }\right) $ is dense in $K_{A_{\lambda
}}(X_{\lambda })$, we deduce that $\psi _{t_{X_{\mu }}\circ \sigma _{\lambda
\mu }^{X}}=\psi _{t_{X_{\mu }}}\circ \pi _{\lambda \mu
}^{L_{A}(X)}|_{K_{A_{\lambda }}(X_{\lambda })}.$

Let $\pi _{\lambda }^{A}\left( a\right) \in J_{X_{\lambda }}$, $a\in
A$. Since $\pi _{\mu }^{A}\left( a\right) =\pi _{\lambda \mu
}^{A}\left( \pi _{\lambda }^{A}\left( a\right) \right) \in J_{X_{\mu
}}$, we have
\begin{eqnarray*}
\psi _{t_{X_{\mu }}\circ \sigma _{\lambda \mu }^{X}}\left( \varphi
_{X_{\lambda }}\left( \pi _{\lambda }^{A}\left( a\right) \right) \right)
&=&\psi _{t_{X_{\mu }}}\left( \pi _{\lambda \mu }^{L_{A}(X)}\left( \varphi
_{X_{\lambda }}\left( \pi _{\lambda }^{A}\left( a\right) \right) \right)
\right) \\
&=&\psi _{t_{X_{\mu }}}\left( \varphi _{X_{\mu }}\left( \pi _{\lambda \mu
}^{A}\left( \pi _{\lambda }^{A}\left( a\right) \right) \right) \right) \\
&=&\psi _{t_{X_{\mu }}}\left( \varphi _{X_{\mu }}\left( \pi _{\mu
}^{A}\left( a\right) \right) \right) =\pi _{X_{\mu }}\left( \pi _{\mu
}^{A}\left( a\right) \right) \\
&=&\pi _{X_{\mu }}\circ \pi _{\lambda \mu }^{A}\left( \pi _{\lambda
}^{A}\left( a\right) \right) .
\end{eqnarray*}

Therefore, the pair $(\pi_{X_{\mu }}\circ \pi _{\lambda \mu }^{A},t_{X_{\mu
}}\circ \sigma _{\lambda \mu }^{X})$ is a covariant representation of the $%
C^{\ast }$-correspon\-dence $\left( X_{\lambda },A_{\lambda },\varphi
_{X_{\lambda }}\right) $ on the $C^{\ast }$-algebra $\mathcal{O}_{X_{\mu }}$%
. From the universality of the covariant representation $(\pi _{X_{\lambda
}},t_{X_{\lambda }})$, there exists a unique $C^{\ast }$-morphism $\rho
_{\lambda \mu }:\mathcal{O}_{X_{\lambda }}\rightarrow \mathcal{O}_{X_{\mu }}$%
, such that $\rho _{\lambda \mu }\circ t_{X_{\lambda }}=t_{X_{\mu }}\circ
\sigma _{\lambda \mu }^{X}$ and $\rho _{\lambda \mu }\circ \pi _{X_{\lambda
}}=\pi _{X_{\mu }}\circ \pi _{\lambda \mu }^{A}$. It is easy to check that $%
\{\mathcal{O}_{X_{\lambda }};\rho _{\lambda \mu };\lambda ,\mu \in \Lambda
,\lambda \geq \mu \}$ is an inverse system of $C^{\ast }$-algebras.
\end{proof}

Using Lemma \ref{steplemma} and following the proof of \cite[Proposition 3.5]%
{JZ}, we obtain the following result, which gives a condition under which
one has a covariant representation of an inverse limit pro-$C^{\ast }$%
-correspondence $(X,A,\varphi _{X})$.

\begin{proposition}
\label{algebraassociated}Let $\left( X\emph{,}A,\varphi _{X}\right) $ be an
inverse limit pro-$C^{\ast }$-correspondence with the property that $\pi
_{\lambda \mu }^{A}\left( J_{X_{\lambda }}\right) \subset J_{X_{\mu }}~$,
for all $\lambda ,\mu \in \Lambda $ with $\lambda \geq \mu $. Then there is
a covariant representation $(\pi ,t)$ of $\left( X\emph{,}A,\varphi
_{X}\right) $ on $\lim\limits_{\leftarrow \lambda }\mathcal{O}_{X_{\lambda
}} $.
\end{proposition}

\begin{proof}
By Lemma \ref{steplemma}, there is a pro-$C^{\ast }$-morphism $\pi
=\lim\limits_{\leftarrow \lambda }\pi _{X_{\lambda }}$ from $A$ to $%
\lim\limits_{\leftarrow \lambda }\mathcal{O}_{X_{\lambda }}$ and a map $%
t=\lim\limits_{\leftarrow \lambda }t_{X_{\lambda }}$ from $X$ to $%
\lim\limits_{\leftarrow \lambda }\mathcal{O}_{X_{\lambda }}$. Following the
proof of \cite[Proposition 3.5]{JZ}, we show that $(\pi ,t)$ is a
representation of $\left( X\emph{,}A,\varphi _{X}\right) $ on $%
\lim\limits_{\leftarrow \lambda }\mathcal{O}_{X_{\lambda }}$. It is easy to
check that $\psi _{t}=\lim\limits_{\leftarrow \lambda }\psi _{t_{X_{\lambda
}}}$. Let $a\in \mathcal{J}_{X}$. Then
\begin{equation*}
\psi _{t}\left( \varphi _{X}\left( a\right) \right) =\left( \psi
_{t_{X_{\lambda }}}\left( \varphi _{X_{\lambda }}\left( \pi _{\lambda
}^{A}\left( a\right) \right) \right) \right) _{\lambda }=\left( \pi
_{X_{\lambda }}\left( \pi _{\lambda }^{A}\left( a\right) \right) \right)
_{\lambda }=\pi \left( a\right) .
\end{equation*}%
Therefore, $(\pi ,t)$ is a covariant representation of $\left( X\emph{,}%
A,\varphi _{X}\right) $ on $\lim\limits_{\leftarrow \lambda }\mathcal{O}%
_{X_{\lambda }}$.
\end{proof}

Next we find out an equivalent form of the condition $\pi _{\lambda \mu
}^{A}\left( J_{X_{\lambda }}\right) \subset J_{X_{\mu }}~$ in Proposition %
\ref{algebraassociated}.

\begin{lemma}
Let $\left( X,A,\varphi _{X}\right) $ be an inverse limit pro-$C^{\ast }$%
-correspondence. Then the following statements are equivalent

\begin{enumerate}
\item[\emph{(1)}] $\pi _{\mu }^{A}\left( J_{X}^{\lambda }\right) \cap \pi
_{\mu }^{A}\left( X^{-1}(\ker p_{\mu })\right) =\{0\}\ $, for all $\lambda
,\mu \in \Lambda $ with $\mu \leq \lambda ;$

\item[\emph{(2)}] $\pi _{\lambda \mu }^{A}\left( J_{X_{\lambda }}\right)
\subset J_{X_{\mu }}~$, for all $\lambda ,\mu \in \Lambda $ with $\mu \leq
\lambda .$
\end{enumerate}
\end{lemma}

\begin{proof}
If $\left( X,A,\varphi _{X}\right) $ is an inverse limit pro-$C^{\ast }$%
-correspondence, then $\varphi _{X}=\lim\limits_{\leftarrow \lambda
}\varphi _{X_{\lambda }}$ and $\pi _{\lambda }^{L_{A}(X)}\circ
\varphi _{X}=\varphi _{X_{\lambda }}\circ \pi _{\lambda }^{A}$, for
all $\lambda \in \Lambda $.

$\left( 1\right) \Rightarrow \left( 2\right) $ Let $\pi _{\lambda
}^{A}\left( a\right) \in J_{X_{\lambda }}$, $a\in A$. Then
\begin{equation*}
\varphi _{X_{\mu }}\left( \pi _{\mu }^{A}\left( a\right) \right) =\pi
_{\lambda \mu }^{L_{A}(X)}\left( \varphi _{X_{\lambda }}\left( \pi _{\lambda
}^{A}\left( a\right) \right) \right) \in \pi _{\lambda \mu
}^{L_{A}(X)}\left( K_{A_{\lambda }}(X_{\lambda })\right) =K_{A_{\mu
}}(X_{\mu }).
\end{equation*}
If $\pi _{\mu }^{A}\left( b\right) \in \ker \varphi _{X_{\mu }}$,
$b\in A$, then
\begin{equation*}
b\in \ker (\varphi _{X_{\mu }}\circ \pi _{\mu }^{A})=\ker (\pi _{\mu
}^{L_{A}(X)}\circ \varphi _{X})=X^{-1}(\ker p_{\mu }).
\end{equation*}%
Therefore,%
\begin{equation*}
\pi _{\mu }^{A}\left( a\right) \pi _{\mu }^{A}\left( b\right) =\pi _{\mu
}^{A}\left( ab\right) \in \pi _{\mu }^{A}\left( J_{X}^{\lambda }\right) \cap
\pi _{\mu }^{A}\left( X^{-1}(\ker p_{\mu })\right) =\{0\},
\end{equation*}%
and so $\pi _{\mu }^{A}\left( a\right) =\pi _{\lambda \mu }^{A}\left( \pi
_{\lambda }^{A}\left( a\right) \right) \in J_{X_{\mu }}.$

$\left( 2\right) \Rightarrow \left( 1\right) $ Let $\pi _{\mu
}^{A}\left( a\right) \in \pi _{\mu }^{A}\left( J_{X}^{\lambda
}\right) \cap \pi _{\mu }^{A}\left( X^{-1}(\ker p_{\mu
})\right),\,a\in A $. Since
\begin{equation*}
\pi _{\mu }^{A}\left( J_{X}^{\lambda }\right) =\pi _{\lambda \mu
}^{A}\left( \pi _{\lambda }^{A}\left( J_{X}^{\lambda }\right)
\right) =\pi_{\lambda \mu}^{A}(J_{X_{\lambda}}) \subseteq J_{X_{\mu
}},
\end{equation*}
 where the second equality is due to Lemma \ref{ideal}, and since
$\pi _{\mu }^{A}\left( X^{-1}(\ker p_{\mu })\right) =\ker \varphi
_{X_{\mu }}$ (see Remark \ref{inv}) we have $\pi _{\mu }^{A}\left(
a\right) \in J_{X_{\mu }}\cap \pi _{\mu }^{A}\left( X^{-1}(\ker
p_{\mu })\right) $ and so $\pi _{\mu }^{A}\left( a\right) =0$.
\end{proof}

\begin{remark}
\emph{Since $\pi _{\mu }^{A}\left( \ker p_{\mu }\right) =\{0\}$ for all $\mu
\in \Lambda ,$ $\pi _{\mu }^{A}\left( J_{X}^{\lambda }\right) \cap \pi _{\mu
}^{A}\left( X^{-1}(\ker p_{\mu })\right) \subset \pi _{\mu }^{A}(\ker p_{\mu
})$ for all $\lambda ,\mu \in \Lambda $ with $\mu \leq \lambda $ if and only
if $\pi _{\mu }^{A}\left( J_{X}^{\lambda }\right) \cap \pi _{\mu }^{A}\left(
X^{-1}(\ker p_{\mu })\right) =\{0\}\ $, for all $\lambda ,\mu \in \Lambda $
with $\mu \leq \lambda .$}
\end{remark}

\begin{remark}
\emph{According to \cite[Definition 4.8]{K3} the condition $X(\ker
p_{\lambda })\subset \ker p_{\lambda }$ set out in Lemma \ref{steplemma1}
can be rephrased as $\ker p_{\lambda }$ is positively invariant for every $%
\lambda \in \Lambda $. Also the condition }$\pi _{\lambda \mu
}^{A}(J_{X_{\lambda }})\subset J_{X_{\mu }}$\emph{\ for all $\lambda ,\mu
\in \Lambda $, with $\,\lambda \geq \mu $, set out in Lemma \ref{steplemma}
resembles to the notion of negative invariance of an ideal given in \cite[%
Definition 4.8]{K3}.}
\end{remark}

\begin{definition}
\label{deffinvariant}\emph{Let $\left( X,A,\varphi _{X}\right) $ be a pro-$%
C^{\ast }$-correspondence.} An ideal \emph{$I$ of $A$} is positively
invariant\emph{\ if $X(I)\subset I$,} negatively invariant \emph{if $\ \pi
_{\mu }^{A}\left( J_{X}^{\lambda }\right) \cap \pi _{\mu }^{A}\left(
X^{-1}(I)\right) \subset \pi _{\mu }^{A}\left( I\right) $, for all $\lambda
,\mu \in \Lambda $ with $\lambda \geq \mu $ and} invariant \emph{if $I$ is
both positively and negatively invariant.}
\end{definition}

According to Definition \ref{deffinvariant}, Lemma \ref{steplemma1}, Lemma %
\ref{steplemma}, and Proposition \ref{algebraassociated} we get the
following result.

\begin{proposition}
\label{conclusion} Let $\left( X,A,\varphi _{X}\right) $ be a pro-$C^{\ast }$%
-correspondence. If $\ \ker p_{\lambda },\lambda \in \Lambda $, are
invariant, then there exists a covariant representation of $\left(
X,A,\varphi _{X}\right) $ on $\lim\limits_{\leftarrow \lambda }\mathcal{O}%
_{X_{\lambda }}.$
\end{proposition}

In order to show in Theorem \ref{PR} below that $\mathcal{O}_{X}$ exists, in
case $X$ is a pro-$C^{\ast }$-correspondence endowed with the property which
is described in Proposition \ref{conclusion}, we are going to use the notion
of a $\mathcal{T}$-pair for a ${C}${$^{\ast }$-correspondence, which was
introduced and studied in \cite[Sections 5-7]{K3}. We recall that given a $%
C^{\ast }$-correspondence $(X,A,\varphi _{X})$, a $\mathcal{T}$-pair of $X$
is a pair $\omega =(I,I^{\prime })$ of ideals $I,I^{\prime }$ of $A$ such
that $X(I)\subset I$ and $I\subset I^{\prime }\subset J(I)=\{a\in A:\phi
_{I}(\varphi _{X}(a))\in K_{A}(X/XI),\,aX^{-1}(I)\subset I\}$ \cite[%
Definition 5.6]{K3} (for the definition of $\phi _{I}$ see Remark \ref{inv}%
). Also for two $\mathcal{T}$-pairs $\omega _{1}=(I_{1},I_{1}^{\prime })$, $%
\omega _{2}=(I_{2},I_{2}^{\prime })$, we denote $\omega _{1}\subset \omega
_{2}$, if $I_{1}\subset I_{2}$ and $I_{1}^{\prime }\subset I_{2}^{\prime }$ 
\cite[Definition 5.7]{K3}. }

Let $\left( X,A,\varphi _{X}\right) $ be a pro-$C^{\ast }$-correspondence
such that $\ker p_{\lambda },\lambda \in \Lambda $ are invariant. For each $%
\lambda \in \Lambda ,$ $\omega _{\lambda }=\left( \{0\},(\mathcal{J}%
_{X})_{\lambda }\right) $ is a $\mathcal{T}$-pair of the $C^{\ast }$%
-correspondence $(X_{\lambda },A_{\lambda },\varphi _{X_{\lambda }})$, since 
\begin{equation*}
(\mathcal{J}_{X})_{\lambda }=\pi _{\lambda }^{A}\left( \mathcal{J}%
_{X}\right) \subset \pi _{\lambda }^{A}\left( J_{X}^{\lambda }\right)
=J_{X_{\lambda }}=J(\{0\}).
\end{equation*}%
Let $\left( \pi _{\omega _{\lambda }},t_{\omega _{\lambda }}\right) $ be the
representation of the $C^{\ast }$-correspondence $\left( X_{\lambda
},A_{\lambda },\varphi _{X_{\lambda }}\right) $ on the $C^{\ast }$-algebra $%
\mathcal{O}_{X_{\omega _{\lambda }}}$associated to the $\mathcal{T}$-pair $%
\omega _{\lambda }$ (see \cite[Definition 6.10]{K3}). Moreover, $\mathcal{O}%
_{X_{\omega _{\lambda }}}$ is generated by the images of $t_{\omega
_{\lambda }}$and $\pi _{\omega _{\lambda }}\ $\cite[Proposition 6.11]{K3}.

Let $\lambda ,\mu \in \Lambda $ with $\lambda \geq \mu $. Then $\left( \pi
_{\omega _{\mu}}\circ \pi _{\lambda \mu }^{A},t_{\omega _{\mu }}\circ \sigma
_{\lambda \mu }^{X} \right) $ is a representation of the $C^{\ast }$%
-correspondence $\left( X_{\lambda },A_{\lambda },\varphi _{X_{\lambda
}}\right) $, and let $\omega _{\left( \pi _{\omega _{\mu }}\circ \pi
_{\lambda \mu }^{A},t_{\omega _{\mu }}\circ \sigma _{\lambda \mu
}^{X}\right) }$ be the $\mathcal{T}$-pair associated to this representation 
\cite[Definition 5.9]{K3}. Then, by definition, 
\begin{equation*}
\omega _{\left( \pi _{\omega _{\mu }}\circ \pi _{\lambda \mu }^{A},t_{\omega
_{\mu }}\circ \sigma _{\lambda \mu }^{X}\right) }=\left( \ker \left(\pi
_{\omega _{\mu }}\circ \pi _{\lambda \mu }^{A} \right) ,\left( \pi _{\omega
_{\mu }}\circ \pi _{\lambda \mu }^{A}\right) ^{-1}\left( \psi _{t_{\omega
_{\mu }}\circ \sigma _{\lambda \mu }^{X}}\left( K_{A_{\lambda }}\left(
X_{\lambda }\right) \right) \right) \right) .
\end{equation*}%
Clearly $\{0\}\subset \ker (\pi _{\omega _{\mu }}\circ \pi _{\lambda \mu
}^{A}),$ and since 
\begin{equation*}
\left( \pi _{\omega _{\mu }}\circ \pi _{\lambda \mu }^{A}\right) \left( (%
\mathcal{J}_{X})_{\lambda }\right) =\pi _{\omega _{\mu }}\left( (\mathcal{J}%
_{X})_{\mu }\right) \subset \psi _{t_{\omega _{\mu }}}\left( K_{A_{\mu
}}\left( X_{\mu }\right) \right) =\psi _{t_{\omega _{\mu }}\circ \sigma
_{\lambda \mu }^{X}}\left( K_{A_{\lambda }}\left( X_{\lambda }\right)
\right) ,
\end{equation*}%
we have $\omega _{\lambda }\subset \omega _{\left( \pi _{\omega _{\mu
}}\circ \pi _{\lambda \mu }^{A},t_{\omega _{\mu }}\circ \sigma _{\lambda \mu
}^{X}\right) }$. On the other hand, 
\begin{equation*}
C^{\ast }\text{-}(t_{\omega _{\mu }}\circ \sigma _{\lambda \mu }^{X}\left(
X_{\lambda }\right) ,\pi _{\omega _{\mu }}\circ \pi _{\lambda \mu
}^{A}\left( A_{\lambda }\right) )=C^{\ast }\text{-}(t_{\omega _{\mu }}\left(
X_{\mu }\right) ,\pi _{\omega _{\mu }}\left( A_{\mu }\right) ),
\end{equation*}
and then, by \cite[Theorem 7.1]{K3}, there exists a unique surjective $%
C^{\ast }$-morphism $\rho_{\lambda \mu }^{\omega }:$ $\mathcal{O}_{X_{\omega
_{\lambda }}}\rightarrow \mathcal{O}_{X_{\omega _{\mu }}}$ such that $\rho
_{\lambda \mu }^{\omega }\circ t_{\omega _{\lambda }}=t_{\omega _{\mu
}}\circ \sigma _{\lambda \mu }^{X}$ and $\rho _{\lambda \mu }^{\omega }\circ
\pi _{\omega _{\lambda }}=\pi _{\omega _{\mu }}\circ \pi _{\lambda \mu }^{A}$%
. It is easy to check that $\{\mathcal{O}_{X_{\omega _{\lambda }}};\rho
_{\lambda \mu }^{\omega };\lambda ,\mu \in \Lambda ,\lambda \geq \mu \}$ is
an inverse system of $C^{\ast }$-algebras.

The following theorem gives a condition under which $\mathcal{O}_{X}$ exists.

\begin{theorem}
\label{PR}Let $\left( X,A,\varphi _{X}\right) $ be a pro-$C^{\ast }$%
-correspondence such that $\ker p_{\lambda },\lambda \in \Lambda $, are
invariant. Then there exists $\mathcal{O}_{X}$. Moreover, $\mathcal{O}%
_{X}=\lim\limits_{\leftarrow \lambda }\mathcal{O}_{X_{\omega _{\lambda }}},$
up to a pro-$C^{\ast }$-isomorphism.
\end{theorem}

\begin{proof}
By the above comments $\left( \pi _{\omega _{\lambda }}\right) _{\lambda }$
is an inverse system of $C^{\ast }$-morphisms and $\left( t_{\omega
_{\lambda }}\right) _{\lambda }$ is an inverse system of linear maps. Let $%
t_{\omega }=\lim\limits_{\leftarrow \lambda }t_{\omega _{\lambda }}$ and $%
\pi _{\omega }=\lim\limits_{\leftarrow \lambda }\pi _{\omega _{\lambda }}$.
Following the proof of \cite[Proposition 3.5]{JZ}, we show that $(\pi
_{\omega },t_{\omega })$ is a representation of $\left( X\emph{,}A,\varphi
_{X}\right) $ on $\lim\limits_{\leftarrow \lambda }\mathcal{O}_{X_{\omega
_{\lambda }}}$. It is easy to check that $\psi _{t_{\omega
}}=\lim\limits_{\leftarrow \lambda }\psi _{t_{\omega _{\lambda }}}$. For $%
a\in \mathcal{J}_{X}$, we have
\begin{eqnarray*}
\psi _{t_{\omega }}\left( \varphi _{X}\left( a\right) \right) &=&\left( \psi
_{t_{\omega _{\lambda }}}\left( \varphi _{X_{\lambda }}\left( \pi _{\lambda
}^{A}\left( a\right) \right) \right) \right) _{\lambda } \\
&&\text{\cite[Lemma 5.10 (v)]{K3}} \\
&=&\left( \pi _{t_{\omega _{\lambda }}}\left( \pi _{\lambda }^{A}\left(
a\right) \right) \right) _{\lambda }=\pi _{\omega }\left( a\right) .
\end{eqnarray*}%
Therefore, $\left( \pi _{\omega },t_{\omega }\right) $ is a covariant
representation of $\left( X\emph{,}A,\varphi _{X}\right) $. Moreover, pro-$%
C^{\ast }$-$(\pi _{\omega }(A),t_{\omega }(X))=\lim\limits_{\leftarrow
\lambda }\mathcal{O}_{X_{\omega _{\lambda }}}$.

Let $\left( \pi,t \right) $ be a covariant representation of $\left( X\emph{,%
}A,\varphi _{X}\right) $ on a pro-$C^{\ast }$-algebra $B$. Then, for each $%
\delta \in \Delta ,$ there exists a representation $\left( \pi _{\delta
},t_{\delta }\right) $ of the $C^{\ast }$-correspondence $\left( X_{\lambda
},A_{\lambda },\varphi _{X_{\lambda }}\right) $ on the $C^{\ast }$-algebra $%
B_{\delta }$ such that $\pi _{\delta }^{B}\circ t=t_{\delta }\circ \sigma
_{\lambda }^{X}$ and $\pi _{\delta }^{B}\circ \pi =\pi _{\delta }\circ \pi
_{\lambda }^{A}.$ Since,
\begin{eqnarray*}
\pi _{\delta }\left( \left( \mathcal{J}_{X}\right) _{\lambda }\right) &=&\pi
_{\delta }^{B}\circ \pi \left( \mathcal{J}_{X}\right) =\pi _{\delta
}^{B}\left( \psi _{t}\left( \varphi _{X}\left( \mathcal{J}_{X}\right)
\right) \right) \subset \pi _{\delta }^{B}\left( \psi _{t}\left( K_{A}\left(
X\right) \right) \right) \\
&=&\psi _{t_{\delta }}\left( \pi _{\lambda }^{L_{A}(X)}\left( K_{A}\left(
X\right) \right) \right) =\psi _{t_{\delta }}\left( K_{A_{\lambda }}\left(
X_{\lambda }\right) \right) ,
\end{eqnarray*}%
$\omega _{\lambda }\subset \omega _{\left( t_{\delta },\pi _{\delta }\right)
}$,$\ $and then, by \cite[Theorem 7.1]{K3}, there exists a surjective $%
C^{\ast }$-morphism $\widetilde{\rho }_{\delta }:\mathcal{O}_{X_{\omega
_{\lambda }}}\rightarrow C^{\ast }$-$(t_{\delta }\left( X_{\lambda }\right)
,\pi _{\delta }\left( A_{\lambda }\right) )$ such that $\widetilde{\rho }%
_{\delta }\circ t_{\omega _{\lambda }}=t_{\delta }$ and $\widetilde{\rho }%
_{\delta }\circ \pi _{\omega _{\lambda }}=\pi _{\delta }$.
Therefore, there is a continuous $\ast $-morphism $\rho _{\delta
}:\lim\limits_{\leftarrow \lambda }\mathcal{O}_{X_{\omega _{\lambda
}}}\rightarrow B_{\delta },$ with $\rho _{\delta }=\widetilde{\rho
}_{\delta }\circ \chi _{\lambda }$, where $\chi _{\lambda }$ is the
canonical projection from $\lim\limits_{\leftarrow \lambda
}\mathcal{O}_{X_{\omega _{\lambda }}}$ to $\mathcal{O}_{X_{\omega
_{\lambda }}}$. For each $\delta _{1},\delta _{2}\in \Delta $, such
that $\delta _{1}\geq \delta _{2}$, we have $\pi _{\delta _{1}\delta
_{2}}^{B}\circ \rho
_{\delta _{1}}=\rho _{\delta _{2}}$ (see the proof of Proposition 3.5 \cite%
{JZ}), and so there is a pro-$C^{\ast }$-morphism $\rho
:\lim\limits_{\leftarrow \lambda }\mathcal{O}_{X_{\omega _{\lambda
}}}\rightarrow B$ such that $\pi _{\delta }^{B}\circ \rho =\rho
_{\delta }$, for all $\delta \in \Delta .$ It is easy to check that
$\rho \circ t_{\omega }=t$ and $\rho \circ \pi _{\omega }=\pi .$
Therefore the
result follows from Definition \ref{universal} and Remark \ref{uniqueness}%
(2).
\end{proof}

\section{ Pro-$C^{\ast }$-correspondences and crossed products of Hilbert
pro-$C^{\ast }$-bimodules}

Let $X$ be a Hilbert bimodule over a pro-$C^{\ast }$-algebra $A$ whose
topology is given by the family of $C^{\ast }$seminorms $\,\{p_{\lambda
},\,\lambda \in \Lambda \}$.

\begin{definition}
\emph{\cite[Definition 3.1]{JZ}} A covariant representation \emph{of a
Hilbert $A-A$ pro-$C^{\ast }$-bimodule $X$ on a pro-$C^{\ast }$-algebra $B$
is a pair $\left( \varphi _{X},\varphi _{A}\right) $ consisting of a pro-$%
C^{\ast }$-morphism $\varphi _{A}:A\rightarrow B$ and a map $\varphi
_{X}:X\rightarrow B$ which verifies the following relations:}

\begin{enumerate}
\item[\emph{(1)}] $\varphi _{X}\left( xa\right) =\varphi _{X}\left( x\right)
\varphi _{A}\left( a\right) $ \emph{and} $\varphi _{X}\left( ax\right)
=\varphi _{A}\left( a\right) \varphi _{X}\left( x\right) $ \emph{for all} $%
x\in X$ \emph{and for all }$a\in A$\emph{$.$}

\item[\emph{(2)}] {{{$\varphi _{X}(x)^{\ast }\varphi _{X}(y)=\varphi
_{A}(<x,y>_{A})$ \emph{and} {$\varphi _{X}(x)\varphi _{X}(y)^{\ast }=\varphi
_{A}(_{A}<x,y>)$ }}}} \emph{for all} $x,y\in X.$
\end{enumerate}
\end{definition}

\begin{definition}
\emph{\cite[Definition 3.3]{JZ}} The crossed product of $A$ by $X$\emph{\ is
a pro-$C^{\ast }$-algebra, denoted by $A\times _{X}\mathbb{Z}$, and a
covariant representation $\left( i_{X},i_{A}\right) $ of $\left( X,A\right) $
on $A\times _{X}\mathbb{Z}$ with the property that for any covariant
representation $\left( \varphi _{X},\varphi _{A}\right) $ of $\left(
X,A\right) $ on a pro-$C^{\ast }$-algebra $B$, there is a unique pro-$%
C^{\ast }$-morphism $\Phi :A\times _{X}\mathbb{Z\rightarrow }B$ such that $%
\Phi \circ i_{X}=\varphi _{X}$ and $\Phi \circ i_{A}=\varphi _{A}$.}
\end{definition}

We will show that the crossed product $A\times _{X}\mathbb{Z}$ of $A$ by $X$
is isomorphic to the pro-$C^{\ast }$-algebra $\mathcal{O}_{X}$ associated to 
$X,$ when $X$ is regarded as a pro-$C^{\ast }$-correspondence.

The following result is a generalization of \cite[Theorem 6.5]{Z}. If $X$ is
a Hilbert $A-A$ pro-$C^{\ast }$-bimodule, then by $_{A}I$, we denote the
closed ideal $\overline{\text{span}}\{_{A}<x,y>:\,x,\,y\in \,X\}$ of $A$.

\begin{proposition}
\label{generalresult} Let $X$ be a Hilbert $A-A$ pro-$C^{\ast }$-bimodule.
Then $_{A}I=K_{A}(X)$, up to a pro-$C^{\ast }$-isomorphism.
\end{proposition}

\begin{proof}
Since $_{A}I$ is a closed $\ast $-ideal of $A$, it is a pro-$C^{\ast }$%
-algebra, hence we get that
\begin{eqnarray*}
_{A}I &=&\lim\limits_{\leftarrow \lambda }\overline{\pi _{\lambda
}^{A}(_{A}I)}=\lim\limits_{\leftarrow \lambda }\overline{\pi _{\lambda }^{A}%
\big(\text{span}\{\,_{A}<x,y>:x,y\,\in \,X\}\big)} \\
&=&\lim\limits_{\leftarrow \lambda }\overline{\text{span}}\{_{A_{\lambda
}}<\sigma _{\lambda }^{X}(x),\sigma _{\lambda }^{X}(y)>:\,x,y\in X\} \\
&=&\lim\limits_{\leftarrow \lambda }\ _{A_{\lambda }}I.
\end{eqnarray*}%
From \cite[Proposition 1.10]{BMS}, we have that for every $\lambda \in
\Lambda $, there exists a $C^{\ast }$-isomorphism $\psi _{\lambda }:\
_{A_{\lambda }}I\rightarrow K_{A_{\lambda }}(X_{\lambda })$ given by
\begin{equation*}
\,\psi _{\lambda }(\pi _{\lambda }^{A}\left( a\right) )(\sigma _{\lambda
}^{X}(x))=\pi _{\lambda }^{A}\left( a\right) \sigma _{\lambda }^{X}(x)
\end{equation*}%
for all $\,a\in\, _{A}I,\,x\in X$. Moreover, for every $a\in
\,_{A}I,\,x\in X,\,\lambda ,\mu \in \Lambda $ with $\lambda \geq \mu
$
\begin{eqnarray*}
((\pi _{\lambda \mu }^{L_{A}(X)}\circ \psi _{\lambda })(\pi _{\lambda
}^{A}(a)))(\sigma _{\mu }^{X}(x)) &=&\sigma _{\lambda \mu }^{X}(\psi
_{\lambda }(\pi _{\lambda }^{A}(a))\sigma _{\lambda }^{X}(x)) \\
&=&\sigma _{\lambda \mu }^{X}(\pi _{\lambda }^{A}(a)\sigma _{\lambda
}^{X}(x))=\sigma _{\mu }^{X}(ax) \\
&=&\psi _{\mu }(\pi _{\mu }^{A}(a))(\sigma _{\mu }^{X}(x)) \\
&=&(\psi _{\mu }\circ \pi _{\lambda \mu }^{A})(\pi _{\lambda
}^{A}(a))(\sigma _{\mu }^{X}(x)).
\end{eqnarray*}%
Therefore $(\psi _{\lambda })_{\lambda \in \Lambda }$ is an inverse system
of $C^{\ast }$-isomorphisms between $_{A_{\lambda }}I$ and $K_{A_{\lambda
}}(X_{\lambda })$. Hence, since $K_{A}(X)=\lim\limits_{\leftarrow \lambda
}K_{A_{\lambda }}(X_{\lambda })$, there is a unique pro-$C^{\ast }$%
-isomorphism $\psi :$ $_{A}I\rightarrow K_{A}(X)$, such that $\psi
(a)(x)=ax$ and $p_{\lambda ,L_{A}\left( X\right) }(\psi
(a))=p_{\lambda }(a)$, $\,$for all $\,\lambda \in \Lambda ,\,x\in
X,\,a\in\, _{A}I$.
\end{proof}

\begin{proposition}
\label{generalresult1} Let $X$ be a Hilbert $A-A$ pro-$C^{\ast }$-bimodule.
If $X$ is viewed as a pro-$C^{\ast }$-correspondence over $A$, then $%
\mathcal{J}_{X}=$ $_{A}I.$
\end{proposition}

\begin{proof}
For each $\lambda \in \Lambda $, we have $\pi _{\lambda }^{A}\left(
J_{X}^{\lambda }\right) =J_{X_{\lambda }}=$ $_{A_{\lambda }}I$ $=\pi
_{\lambda }^{A}$ $\left( _{A}I\right)$ (for the equality $%
J_{X_{\lambda}}=_{A_{\lambda}}I$ see \cite[Lemma 2.4]{K1}). Then
$a\in \mathcal{J}_{X}$ if and only if $\pi _{\lambda }^{A}\left(
a\right) \in \pi _{\lambda }^{A}\left( J_{X}^{\lambda }\right) =\pi
_{\lambda }^{A}$ $\left( _{A}I\right) $, for all $\lambda \in
\Lambda $, that is if and only if $a\in\, _{A}I.$
\end{proof}

Then from Proposition \ref{generalresult1} and Proposition \ref%
{generalresult}, we get the following corollary.

\begin{corollary}
\label{coincidence in the bimodule case} Let $X$ be a Hilbert $A-A$ pro-$%
C^{\ast }$-bimodule. If $X$ is viewed as a pro-$C^{\ast }$-correspondence
over $A$, then $\mathcal{J}_{X}=K_{A}(X)$, up to a pro-$C^{\ast }$%
-isomorphism.\textit{\newline
Moreover, the pro-}$C^{\ast }$-isomorphism from $\mathcal{J}_{X}$ to $%
K_{A}(X)$ is given by $\Psi :\mathcal{J}_{X}$ $\rightarrow $ $K_{A}(X),$ $%
\Psi \left( a\right) x=ax.$
\end{corollary}

\begin{proposition}
\label{correspondences} Let $\left( X,A,\varphi _{X}\right) $ be a pro-$%
C^{\ast }$-correspondence. Then the following assertions are equivalent:

\begin{enumerate}
\item[\emph{(1)}] $X$ has the structure of a Hilbert $A-A$ pro-$C^{\ast }$-
bimodule;

\item[\emph{(2)}] $\varphi _{X}|_{\mathcal{J}_{X}}$ is a pro-$C^{\ast }$%
-isomorphism onto $K_{A}(X)$ such that $p_{\lambda ,L_{A}(X)}\left( \varphi
_{X}\left( a\right) \right) =p_{\lambda }\left( a\right) $, for all $a\in 
\mathcal{J}_{X},\,\lambda \in \Lambda $.
\end{enumerate}
\end{proposition}

\begin{proof}
$\left( 1\right) \Rightarrow \left( 2\right) $ If follows from Corollary \ref%
{coincidence in the bimodule case}.

$\left( 2\right) \Rightarrow \left( 1\right) $ It is easy to check
that $X$ has the structure of a left $A$-module with $ax=\varphi
_{X}(a)(x),$ $a\in
A,x\in X$ and $_{A}\left\langle x,y\right\rangle =\left( \varphi _{X}|_{%
\mathcal{J}_{X}}\right) ^{-1}\left( \theta _{x,y}\right) ,$ $x,y\in
X$, defines a left inner product on $X$. To show that $X$ is a
Hilbert $A-A$ bimodule, it remains to prove the coincidence of the
topologies inherited on $X$\ by the two inner products. For all
$x\in X$\ and $\lambda\in\Lambda$, we have
\begin{eqnarray*}
^{A}p_{\lambda }(x)^{2} &=&p_{\lambda }(_{A}<x,x>)=p_{\lambda }((\varphi
_{X}|_{\mathcal{J}_{X}})^{-1}(\theta _{x,x})) \\
&=&p_{\lambda ,L_{A}(X)}\left( \theta _{x,x}\right) =p_{\lambda }\left(
\left\langle x,x\right\rangle _{A}\right) =p_{\lambda }^{A}(x)^{2}.
\end{eqnarray*}
\end{proof}

\begin{remark}
\emph{In case $\left( X,A,\varphi _{X}\right) $ is an inverse limit pro-$%
C^{\ast }$-correspondence and $\varphi _{X}|_{\mathcal{J}_{X}}$ is a pro-$%
C^{\ast }$-isomorphism onto $K_{A}(X)$, then $p_{\lambda
,\,L_{A}(X)}(\phi_{X}(a))=p_{\lambda }(a)$, $\,\ $for$\ $all $\,a\in \,%
\mathcal{J}_{X}$ and $\,\lambda \in \Lambda $. Indeed, since $\left(
X,A,\varphi _{X}\right) $ is an inverse limit pro-$C^{\ast }$%
-correspondence, $\varphi _{X}=\lim\limits_{\leftarrow \lambda }\varphi
_{X_{\lambda }}$, and it is easy to check that $\pi _{\lambda
}^{L_{A}(X)}\circ \varphi _{X}|_{\mathcal{J}_{X}}=\varphi _{X_{\lambda
}}|_{\left( \mathcal{J}_{X}\right) _{\lambda }}\ $for each $\lambda \in
\Lambda $. Let $\lambda \in \Lambda $. We will show that $\varphi
_{X_{\lambda }}|_{\left( \mathcal{J}_{X}\right) _{\lambda }}:\left( \mathcal{%
J}_{X}\right) _{\lambda }\rightarrow K_{A_{\lambda }}(X_{\lambda })$ is a $%
C^{\ast }$-isomorphism. Then it will follow that} 
\begin{eqnarray*}
p_{\lambda ,\,L_{A}(X)}(\varphi _{X}(a)) &=&||\pi _{\lambda
}^{L_{A}(X)}(\varphi _{X}(a))||=||\varphi _{X_{\lambda }}(\pi _{\lambda
}^{A}(a))|| \\
&=&||\pi _{\lambda }^{A}(a)||=p_{\lambda }(a)
\end{eqnarray*}%
\emph{for all $a\in \mathcal{J}_{X}$. So, let $b\in \mathcal{J}_{X}$, such
that $\varphi _{X_{\lambda }}(\pi _{\lambda }^{A}\left( b\right) )=0$. Then $%
b\in \ker (\pi _{\lambda }^{L_{A}(X)}\circ \varphi _{X})$ and therefore $%
b^{\ast }\in \ker (\pi _{\lambda }^{L_{A}(X)}\circ \varphi _{X})$. Since $%
b\in \mathcal{J}_{X}$ we have $\pi _{\lambda }^{A}(b)\pi _{\lambda
}^{A}(b^{\ast })=0$ and then $p_{\lambda }(b)^{2}=p_{\lambda }(bb^{\ast })=0.
$ Therefore, $\pi _{\lambda }^{A}(b)=0$ and thus $\varphi _{X_{\lambda
}}|_{\left( \mathcal{J}_{X}\right) _{\lambda }}$ is injective. Furthermore $%
\varphi _{X_{\lambda }}|_{\left( \mathcal{J}_{X}\right) _{\lambda }}$ is
surjective, since 
\begin{eqnarray*}
\varphi _{X_{\lambda }}\left( \left( \mathcal{J}_{X}\right) _{\lambda
}\right) &=&\varphi _{X_{\lambda }}\left( \pi _{\lambda }^{A}\left( \mathcal{%
J}_{X}\right) \right) =\pi _{\lambda }^{L_{A}(X)}\left( \varphi _{X}\left( 
\mathcal{J}_{X}\right) \right) \\
&=&\pi _{\lambda }^{L_{A}(X)}\left( K_{A}(X)\right) =K_{A_{\lambda
}}(X_{\lambda }).
\end{eqnarray*}%
}
\end{remark}

\begin{remark}
\label{1}\emph{Let $X$ be a Hilbert $A-A$ pro-$C^{\ast }$-bimodule. If $X$
is regarded as a pro-$C^{\ast }$-correspondence, then, for each $\lambda \in
\Lambda $, we have $(\mathcal{J}_{X})_{\lambda }=\pi _{\lambda }^{A}\left( 
\mathcal{J}_{X}\right) =\pi _{\lambda }^{A}\left( _{A}I\right) =\
_{A_{\lambda }}I=J_{X_{\lambda }}.$}

\emph{Let $\lambda \in \Lambda $. Since }$\left( \pi _{X_{\lambda
}},t_{X_{\lambda }}\right) $\emph{\ is an injective covariant representation
of $X_{\lambda }\ $which admits a gauge action and}%
\begin{equation*}
\omega _{\lambda }=\left( \{0\},(\mathcal{J}_{X})_{\lambda }\right) =\left(
\{0\},J_{X_{\lambda }}\right) =\omega _{\left( \pi _{X_{\lambda
}},t_{X_{\lambda }}\right) ,}
\end{equation*}%
\emph{by \cite[Theorem 7.1]{K3}, there is a unique $C^{\ast }$-isomorphism }$%
\rho _{\lambda }:\mathcal{O}_{X_{\omega _{\lambda }}}$\emph{\ }$\rightarrow 
\mathcal{O}_{X_{\lambda }}$\emph{$\emph{\ such}$\ that }$\rho _{\lambda
}\circ t_{\omega _{\lambda }}=t_{X_{\lambda }}$ \emph{\ and }$\rho _{\lambda
}\circ \pi _{\omega _{\lambda }}=\pi _{X_{\lambda }}$\emph{.}

\emph{On the other hand, by \cite[Proposition 3.7]{K1}, $\mathcal{O}%
_{X_{\lambda }}$ is canonically isomorphic to the crossed product $%
A_{\lambda }\times _{X_{\lambda }}\mathbb{Z}$ of $A_{\lambda }$ by }$%
X_{\lambda }$\emph{. Therefore, the $C^{\ast }$-algebras $\mathcal{O}%
_{X_{\omega _{\lambda }}}$ and $A_{\lambda }\times _{X_{\lambda }}\mathbb{Z}$
are canonically isomorphic.}
\end{remark}

Based on \cite[Proposition 3.8]{JZ}, Remark \ref{1} and Theorem \ref{PR} we
have the following.

\begin{proposition}
Let $X$ be a Hilbert $A-A$ pro-$C^{\ast }$-bimodule. Then the pro-$C^{\ast }$%
-algebras $\mathcal{O}_{X}$ and $A\times _{X}\mathbb{Z}$ are isomorphic,
when $X$ is regarded as a pro-$C^{\ast }$-correspondence.
\end{proposition}

\section{Pro-$C^{\ast }$-correspondences and pro-$C^{\ast }$-crossed
pro\-ducts by automorphisms}

Let $A$ be a pro-$C^{\ast }$-algebra whose topology is given by the family
of $C^{\ast }$-seminorms $\{p_{\lambda };\lambda \in \Lambda \}\ $and $%
\alpha $ a strong bounded automorphism of $A$ (that is, for each $\lambda
\in \Lambda $, there is $\mu $ $\in \Lambda \ $such that $p_{\lambda }\left(
\alpha ^{n}\left( a\right) \right) \leq p_{\mu }\left( a\right) $ for all $%
a\in A$ and for all integers $n)$. We will show that the pro-$C^{\ast }$%
-algebra $\mathcal{O}_{A}$ associated to the pro-$C^{\ast }$-correspondence $%
(A,A,\varphi _{A})\ $ (see Example \ref{automorphism}) and $A\times _{\alpha
}\mathbb{Z}$, the crossed product of $A$ by $\alpha $, are isomorphic as pro-%
$C^{\ast }$-algebras.

Indeed, if $\alpha $ is an automorphism of $A$ as above, then $\left(
A,\alpha ,\mathbb{Z}\right) $ is a pro-$C^{\ast }$-dynamical system with the
action of $\mathbb{Z}$ on $A$ given by $n\rightarrow \alpha ^{n}$,$\ $and $%
A\times _{\alpha }\mathbb{Z}$ is the universal pro-$C^{\ast }$-algebra with
respect to the nondegenerate covariant representations of$\ \left( A,\alpha ,%
\mathbb{Z}\right) \ $\cite[Definition 5.4 and Theorem 5.9]{J3}.

If $\left( u,\varphi \right) \ $ is a nondegenerate covariant representation
of $\left( A,\alpha ,\mathbb{Z}\right) $ on a pro-$C^{\ast }$-algebra $B$,
then $\left(\pi, t \right) $, where $\pi $ $=\varphi $ and $t\left( a\right)
=u_{1}^{\ast }\varphi \left( a\right) $ is a nondegenerate representation of 
$(A,A,\varphi _{A})$ on $B$. Moreover, this representation is covariant.
Indeed, since $K_{A}(A)=A$, the pro-$C^{\ast }$-morphism $\psi _{t}$ is
given by $\psi _{t}\left( a\right) =u_{1}^{\ast }\varphi \left( a\right)
u_{1}$, and then 
\begin{eqnarray*}
\psi _{t}\left( \varphi _{A}\left( a\right) \right) &=&u_{1}^{\ast }\varphi
\left( \varphi _{A}\left( a\right) \right) u_{1}=u_{1}^{\ast }\varphi \left(
\alpha \left( a\right) \right) u_{1} \\
&=&u_{1}^{\ast }u_{1}\varphi \left( a\right) u_{1}^{\ast }u_{1}=\varphi
\left( a\right) =\pi \left( a\right),
\end{eqnarray*}%
for all $a\in \mathcal{J}_{A}$.

Conversely, if $\left( \pi,t \right) $ is a nondegenerate covariant
representation of $(A,A,\varphi _{A})$ on a pro-$C^{\ast }$-algebra $B$,
then the map $u:B\rightarrow B$ defined by $u\left( t\left( a\right)
b\right) =\pi \left( a\right) b$ is a unitary operator, and $\left(
u,\varphi \right) $, where $\varphi =\pi $ and $n\rightarrow u_{n}=u^{n}$
with $u_{0}=$id$_{B}$, is a nondegenerate covariant representation of $%
\left( A,\alpha ,\mathbb{Z}\right) $ on $B$.

We remark that if $\left( \pi,t \right) $ is a covariant representation of a
nondegenerate pro-$C^{\ast }$-correspondence $\left( X,A,\varphi _{X}\right) 
$ on a pro-$C^{\ast }$-algebra $B$, then $\left( \pi,t \right) $ is a
nondegenerate covariant representation of $\left( X,A,\varphi _{X}\right) $
on the pro-$C^{\ast }$-algebra pro-$C^{\ast }$-$\{t(X),\pi \left( A\right)
\}.$

Using these facts and the universal property for crossed products of pro-$%
C^{\ast }$-algebras \cite[Corollary 5.7]{J3}, we have the following
proposition.

\begin{proposition}
Let $A$ be a pro-$C^{\ast }$-algebra, whose topology is given by the family
of $C^{\ast }$-seminorms $\{p_{\lambda };\lambda \in \Lambda \}\ $and let $%
\alpha $ be an automorphism of $A$ with the property that for each $\lambda
\in \Lambda $, there is $\mu $ $\in \Lambda \ $such that $p_{\lambda }\left(
\alpha ^{n}\left( a\right) \right) \leq p_{\mu }\left( a\right) $, for all $%
a\in A$ and for all integers $n$. Then the pro-$C^{\ast }$-algebras \ $%
\mathcal{O}_{A} $ and $A\times _{\alpha }\mathbb{Z}$ are isomorphic.
\end{proposition}

\noindent \textbf{Acknowledgements. } The authors would like to thank
Professor M. Fragoulo\-poulou of the University of Athens, for valuable
comments and suggestions during the preparation of this paper. The first
author was supported by the grant of the Romanian Ministry of Education,
CNCS - UEFISCDI, project number PN-II-ID-PCE-2012-4-0201.

%%%%%%%%%%%%%%%%%%%%%%%%%%%%%%%%%%%%%%%%%%%%%%%%%%%%%%%%%%%%%%%%%%%%%%%%%%%%%%%%%%%%%%%%%%%%%%%%%%%

\begin{flushleft}
Department of Mathematics, Faculty of Applied Sciences, University
Politehnica of Bucharest, 313 Spl. Independentei Street, 060042, Bucharest,
Romania and Simion Stoilow Institute of Mathematics of the Roumanian
Academy, 21 Calea Grivitei Street, 010702, Bucharest, Romania

mjoita@fmi.unibuc.ro \newline
http://sites.google.com/a/g.unibuc.ro/maria-joita/ \newline

Department of Mathematics, University of Athens, Panepistimiopolis, Athens
15784, Greece

gzarak@math.uoa.gr
\end{flushleft}


\bigskip

\begin{thebibliography}{AEE}
\bibitem[AEE]{AEE} B. Abadie, S. Eilers, R. Exel, Morita equivalence for
crossed products by Hilbert $C^{\ast }$-bimodules, \textit{Trans. Amer.
Math. Soc.} 350 (1998), 3043-3054.

\bibitem[BMS]{BMS} L.G.~Brown, J.A.~Mingo, N.-T.~Shen, \textit{\ }%
Quasi-multipliers and embeddings of Hilbert $C^{\ast }$-bimodules, \textit{%
Canad. J. Math.} (6), 46(1994), 1150--1174.

\bibitem[F]{F} M. Fragoulopoulou, \textit{Topological algebras with
involution}, North-Holland Mathematics Studies, 200. Elsevier Science B.V.,
Amsterdam, 2005.

\bibitem[I]{I} A. Inoue, Locally $C^{\ast }$-algebras, \textit{Mem. Fac.
Sci. Kyushu Univ. Ser. A} 25(1971), 197--235.

\bibitem[J1]{J1} M. Joi\c{t}a, \textit{Hilbert modules over locally }$%
C^{\ast }$\textit{-algebras}, Bucharest University Press, Bucharest 2006.

\bibitem[J2]{J2} M. Joi\c{t}a, \textit{Crossed products of locally }$C^{\ast
}$\textit{-algebras}, Editura Academiei Rom\^{a}ne, Bucharest, 2007.

\bibitem[J3]{J3} M. Joi\c{t}a, \textit{Group actions on pro-}$C^{\ast }$%
\textit{-algebras and their pro-}$C^{\ast }$\textit{-crossed products,}
arXiv:1401.2903 [math.OA]

\bibitem[JZ]{JZ} M. Joi\c{t}a, I. Zarakas, Crossed products by Hilbert pro-$%
C^{\ast }$-bimodules, \textit{Studia Math.} 215(2013),2, 139-156.

\bibitem[KPW]{KPW} T. Kajiwara, C. Pinzari, Y. Watawani, Ideal structure and
simplicity of the $C^{\ast }$-algebras generated by Hilbert bimodules,%
\textit{\ J. Funct. Anal.} 159(1998), 295--322.

\bibitem[K1]{K1} T.~Katsura, \textit{A construction of }$C^{\ast }$\textit{\
-algebras from }$C^{\ast }$\textit{-correspondences,} Advances in quantum
dynamics, 173--182, Contemp. Math., 335, Amer. Math. Soc., Providence, RI,
2003.

\bibitem[K2]{K2} T.~Katsura, On $C^{\ast }$-algebras associated with $%
C^{\ast }$-correspondences\textit{,} \textit{J. Funct. Anal.} 217 (2004),
366--401.

\bibitem[K3]{K3} T.~Katsura, Ideal structure of $C^{\ast }$-algebras
associated with $C^{\ast }$-correspondences, \textit{Pacific J. Math.}
230(2007), 107--145.

\bibitem[M]{M} A. Mallios, \textit{Topological Algebras.} Selected Topics,
North-Holland, Amsterdam, 1986.

\bibitem[Ph]{Ph} N.C. Phillips, Inverse limits of $C^{\ast }$-algebras, 
\textit{J. Operator Theory} 19(1988), 159--195.

\bibitem[Pi]{Pi} M.V. Pimsner, \textit{A class of C$^{*}$-algebras
generalizing both Cuntz-Krieger algebras and crossed products by $\mathbb{Z}$%
}, Free probability theory, 189-212, Fields Inst. Commun., 12, Amer. Math.
Soc., Providence, RI, 1997.

\bibitem[Z]{Z} I.~Zarakas, Hilbert pro-$C^{\ast }$-bimodules and
applications, \textit{Rev. Roumaine Math. Pures Appl.} 57(2012), 289--310.
\end{thebibliography}
\end{document}